\documentclass[11pt]{article}

\usepackage{epsfig,epsf,fancybox}
\usepackage{amsmath}
\usepackage{mathrsfs}
\usepackage{amssymb}
\usepackage{graphicx}
\usepackage{color}
\usepackage{multirow}
\usepackage{paralist}
\usepackage{verbatim}
\usepackage{galois}
\usepackage{algorithm}
\usepackage{algorithmic}
\usepackage{boxedminipage}
\usepackage{booktabs}
\usepackage{accents}
\usepackage{stmaryrd}
\usepackage{tikz}
\usetikzlibrary{calc}
\usepackage{subcaption}
\usepackage[bottom]{footmisc}
\usepackage{natbib}
\usepackage{tablefootnote}
\usepackage[colorlinks,linkcolor=magenta,citecolor=blue, pagebackref=true,backref=true]{hyperref}
\renewcommand*{\backrefalt}[4]{%
    \ifcase #1 \footnotesize{(Not cited.)}%
    \or        \footnotesize{(Cited on page~#2.)}%
    \else      \footnotesize{(Cited on pages~#2.)}%
    \fi}

\long\def\comment#1{}
\usepackage{nicefrac}

\newtheorem{theorem}{Theorem}[section]

\newtheorem{lemma}[theorem]{Lemma}
\newtheorem{proposition}[theorem]{Proposition}

\newtheorem{definition}{Definition}[section]

\newtheorem{remark}{Remark}[section]
\newtheorem{assumption}{Assumption}[section]

\numberwithin{equation}{section}

\newcommand{\BB}{\mathbb{B}}
\newcommand{\EE}{\mathbb{E}}

\newcommand{\conv}{\textnormal{conv}}

\newcommand{\algo}{\textsc{A}}
\newcommand{\bound}{\textsc{T}}
\newcommand{\zr}{\textnormal{zr}}

\newcommand{\supp}{\textnormal{supp}\,}

\newcommand{\x}{\mathbf x}
\newcommand{\y}{\mathbf y}

\newcommand{\g}{\mathbf g}

\newcommand{\z}{\mathbf z}

\newcommand{\su}{\mathbf u}
\newcommand{\sv}{\mathbf v}

\newcommand{\argmin}{\mathop{\rm{argmin}}}
\newcommand{\argmax}{\mathop{\rm{argmax}}}

\newcommand{\TCal}{\mathcal{T}}
\newcommand{\XCal}{\mathcal{X}}

\newcommand{\one}{\mathbf{1}}
\newcommand{\zero}{\mathbf{0}}

\newcommand{\br}{\mathbb{R}}

\newcommand{\bo}{\mathbb{O}}

\newcommand{\ACal}{\mathcal{A}}

\newcommand{\FCal}{\mathcal{F}}

\newcommand{\norm}[1]{\|{#1} \|}

\newcommand{\calC}{\mathcal{C}}
\newcommand{\softmax}{\mathrm{softmax}}

\begin{document}

\begin{center}

{\bf{\LARGE{On the Complexity of Deterministic Nonsmooth and \\ [.2cm] Nonconvex Optimization}}}

\vspace*{.2in}
{\large{
\begin{tabular}{ccc}
Michael I. Jordan$^{\diamond, \dagger}$ & Tianyi Lin$^\diamond$ & Manolis Zampetakis$^\diamond$
\end{tabular}
}}

\vspace*{.2in}

\begin{tabular}{c}
Department of Electrical Engineering and Computer Sciences$^\diamond$ \\
Department of Statistics$^\dagger$ \\ 
University of California, Berkeley
\end{tabular}

\vspace*{.2in}

\today

\vspace*{.2in}

\begin{abstract}
In this paper, we present several new results on minimizing a nonsmooth and nonconvex function under a Lipschitz condition. Recent work shows that while the classical notion of Clarke stationarity is computationally intractable up to some sufficiently small constant tolerance, the randomized first-order algorithms find a $(\delta, \epsilon)$-Goldstein stationary point with the complexity bound of $\tilde{O}(\delta^{-1}\epsilon^{-3})$, which is independent of dimension $d \geq 1$~\citep{Zhang-2020-Complexity, Davis-2022-Gradient, Tian-2022-Finite}. However, the deterministic algorithms have not been fully explored, leaving open several problems in nonsmooth nonconvex optimization. Our first contribution is to demonstrate that the randomization is \textit{necessary} to obtain a dimension-independent guarantee, by proving a lower bound of $\Omega(d)$ for any deterministic algorithm that has access to both $1^\textnormal{st}$ and $0^\textnormal{th}$ oracles. Furthermore, we show that the $0^\textnormal{th}$ oracle is \textit{essential} to obtain a finite-time convergence guarantee, by showing that any deterministic algorithm with only the $1^\textnormal{st}$ oracle is not able to find an approximate Goldstein stationary point within a finite number of iterations up to sufficiently small constant parameter and tolerance. Finally, we propose a deterministic smoothing approach under the \textit{arithmetic circuit} model where the resulting smoothness parameter is exponential in a certain parameter $M > 0$ (e.g., the number of nodes in the representation of the function), and design a new deterministic first-order algorithm that achieves a dimension-independent complexity bound of $\tilde{O}(M\delta^{-1}\epsilon^{-3})$.  
\end{abstract}

\end{center}

\section{Introduction}
Fixing $d \geq 1$, we define $\br^d$ as a finite-dimensional Euclidean space and assume the function $f: \br^d \mapsto \br$ is $L$-Lipschitz ($|f(\x) - f(\x')| \leq L\|\x - \x'\|$ for all $\x, \x' \in \br^d$) and satisfies  $\min_{\x \in \br^d} f(\x) > -\infty$. Then, our problem of interest is the following nonsmooth and nonconvex optimization problem:
\begin{equation}\label{prob:main}
\min_{\x \in \br^d} f(\x). 
\end{equation}
The theoretical analysis of nonsmooth and nonconvex functions has long been a focus of mathematical research in economics, control theory and computer science~\citep{Clarke-1990-Optimization, Makela-1992-Nonsmooth, Outrata-1998-Nonsmooth}. In recent years, these results have received renewed attention due to their increasing relevance to machine learning, with applications including the training of neural networks with rectified linear units (ReLUs)~\citep{Nair-2010-Rectified, Glorot-2011-Deep}. 

This line of research on nonsmooth nonconvex optimization dates to the introduction of generalized gradients~\citep{Clarke-1974-Necessary, Clarke-1975-Generalized, Clarke-1981-Generalized} and has been subsequently developed from a variety of viewpoints in the optimization community~\citep{Clarke-2008-Nonsmooth, Rockafellar-2009-Variational, Burke-2020-Gradient}. Roughly speaking, the notion of generalized gradients is a natural extension of gradients for smooth optimization and subdifferentials for nonsmooth convex optimization. It is also referred to as the Clarke subdifferential or the set of Clarke subgradients in the literature.

Given the computational intractability of globally minimizing a Lipschitz function $f$ up to a small constant tolerance~\citep{Nemirovski-1983-Problem, Murty-1987-Some, Nesterov-2018-Lectures}, it is natural to turn to considering the problem of finding an $\epsilon$-Clarke stationary point of $f$, meaning some $\x \in \br^d$ such that $\min\{\|\g\|: \g \in \partial f(\x)\} \leq \epsilon$.  For this, the notion of generalized gradients is a key technique. It has been used to define the $\epsilon$-steepest descent direction for a Lipschitz function~\citep{Goldstein-1977-Optimization}, providing the basis for the Goldstein's subgradient method. Subsequently, a wide range of optimization algorithms have been proposed and shown to achieve the \textit{asymptotic} convergence to a Clarke stationary point under some certain regularity conditions~\citep{Kiwiel-1996-Restricted, Burke-2002-Approximating, Burke-2002-Two, Benaim-2005-Stochastic,Burke-2005-Robust, Kiwiel-2007-Convergence, Davis-2020-Stochastic}. However, it is worth remarking that the standard subgradient method fails to asymptotically approach a Clarke stationary point of a Lipschitz function in general~\citep{Daniilidis-2020-Pathological}. For an overview of these aforementioned algorithms for nonsmooth nonconvex optimization and relevant theoretical results, we refer to Section~\ref{sec:related_work}.

A further question is to ask if there exists an algorithm that can achieve finite-time convergence to an $\epsilon$-Clarke stationary point of a Lipschitz function. A \textit{negative} answer to this question was formally stated in~\citet[Theorem~1]{Zhang-2020-Complexity}: for any algorithm that has access to both a function value and a generalized gradient at each query point, there are one-dimensional functions $f_1$ and $f_2$ that do not share any common $\epsilon$-Clarke stationary point for a fixed $\epsilon \in [0, 1)$, such that these two functions can not be distinguished using a finite number of queries. Moreover,  finding a \textit{near} $\epsilon$-Clarke stationary point of a Lipschitz function is proven to be impossible unless the number of queries has an exponential dependence on $d$~\citep{Kornowski-2021-Oracle}. These negative results suggest the need for rethinking the definition of targeted stationary points and motivates us to formulate our goals in terms of a more relaxed yet still meaningful notion of so-called Goldstein subdifferential (or a set of Goldstein stationary points)~\citep{Goldstein-1977-Optimization}. In particular, we consider the problem of finding a $(\delta, \epsilon)$-Goldstein stationary point of $f$ in Eq.~\eqref{prob:main}, meaning some $\x \in \br^d$ such that
\begin{equation}\label{prob:Goldstein}
\min\left\{\|\g\|: \g \in \partial_\delta f(\x)\right\} \leq \epsilon, 
\end{equation}
where the Goldstein subdifferential $\partial_\delta f(\cdot)$ at a point $\x \in \br^d$ is defined by\footnote{We let $\textnormal{conv}(S)$ denote the convex hull of $S$.}
\begin{equation*}
\partial_\delta f(\x) := \textnormal{conv}\left(\cup_{\y: \|\y - \x\| \leq \delta}  \partial f(\y)\right). 
\end{equation*}
This is a fundamental yet largely unexplored research topic. In this context,~\citet{Zhang-2020-Complexity} was the first to propose the randomized variant of Goldstein's subgradient method for minimizing a Hadamard directionally differentiable function and proved a complexity bound of $O(\delta^{-1}\epsilon^{-3})$ in terms of $1^\textnormal{st}$ and $0^\textnormal{th}$ oracles. Subsequently,~\citet{Davis-2022-Gradient} focused on a class of Lipschitz functions and proposed another randomized variant that achieved the same theoretical guarantee. Concurrently,~\citet{Tian-2022-Finite} removed the subgradient selection oracle assumption in~\citet[Assumption~1]{Zhang-2020-Complexity} and provided the third randomized variant of Goldstein's subgradient method that achieved the same complexity guarantee. Very recently,~\citet{Lin-2022-Gradient} have developed a bunch of gradient-free method for minimizing a general Lipschitz function and proved that they yield a complexity bound of $O(d^{3/2}\delta^{-1}\epsilon^{-4})$ in terms of (noisy) $0^\textnormal{th}$ oracles. 

All of the aforementioned algorithms are randomized and their complexity bounds are independent of problem dimension $d \geq 1$ if we are accessible to both $1^\textnormal{st}$ and $0^\textnormal{th}$ oracles. However, it still remains unknown what are lower and upper bounds for finding a $(\delta, \epsilon)$-Goldstein stationary point in deterministic nonsmooth and nonconvex optimization. Accordingly, we have the following open question: 
\begin{quote}
\centering
\textit{Is randomization necessary to obtain the dimension-independent complexity guarantee?} 
\end{quote}
Moreover, the existing algorithms require the access to both $1^\textnormal{st}$ and $0^\textnormal{th}$ oracles due to the use of line search scheme. This is in contrast to smooth nonconvex optimization where the $1^\textnormal{st}$ oracle suffices to ensure the dimension-independent complexity guarantee (e.g., the gradient descent scheme).  This raises another open question: 
\begin{quote}
\centering
\textit{Is the $0^\textnormal{th}$ oracle necessary to obtain the dimension-independent complexity guarantee, or even a finite-time convergence guarantee?} 
\end{quote}

\subsection{Our Contribution} \label{sec:contribution}
This paper presents some new lower and upper bounds on the complexity of finding a $(\delta, \epsilon)$-Goldstein stationary point using deterministic algorithms, thus answering the aforementioned two open problems in the affirmative. Our contributions can be summarized as follows: 
\begin{enumerate}
\item \textbf{Necessity of randomness for dimension-independent complexity bounds (Theorem~\ref{Theorem:DET}).} We prove a dimension-dependent lower bound of $\Omega(d)$ for any deterministic algorithm that has the access to $1^\textnormal{st}$ and $0^\textnormal{th}$ oracles for finding a $(\delta, \epsilon)$-Goldstein stationary point as $\delta, \epsilon > 0$ are small. 

\item \textbf{Necessity of $0^\textnormal{th}$ oracle for deterministic algorithms (Theorem~\ref{Theorem:1LB}).} We prove that the $0^\textnormal{th}$ oracle is essential to obtain a finite-time convergence guarantee by showing that any deterministic algorithm with only the $1^\textnormal{st}$ oracle can not find a Goldstein stationary point within a finite number of iterations up to some small constant tolerance. Moreover, we clarify why the $1^\textnormal{st}$ oracle suffices to obtain finite-time guarantee for randomized algorithms (see Remark \ref{rem:randomized1stOracle}). 
\item \textbf{Deterministic smoothing and algorithm with logarithmic dependence on smoothness (Theorem~\ref{thm:circuitSmoothing} and~\ref{Theorem:SMOOTH}).} We propose a deterministic smoothing approach that induces a smoothness parameter which is exponential in $M > 0$ given that we have a proper arithmetic circuit model that characterizes the structure of $f$. Focusing on such optimization problem with the smoothness parameter $\Theta(2^M)$, we develop a deterministic algorithm with a dimension-independent complexity bound of $\tilde{O}(M\delta^{-1}\epsilon^{-3})$ in terms of both $1^\textnormal{st}$ and $0^\textnormal{th}$ oracles. 
\end{enumerate}
With the aforementioned results, our current knowledge about the complexity of nonsmooth nonconvex optimization is summarized in Table \ref{tbl:summary}.
\begin{table}
\renewcommand{\arraystretch}{2}
\caption{Summary of the theoretical guarantees that are known for nonsmooth nonconvex optimization under a Lipschitz condition. We present the detailed comparison between deterministic algorithms and randomized algorithms and also highlight the key role of $0^\textnormal{th}$ oracle. Here, \textbf{U} stands for upper bounds, i.e., complexities that achieved by existing algorithms, and \textbf{L} stands for lower bounds.}
\label{tbl:summary}
\centering
\begin{tabular}{l|ll|ll} \hline
\textbf{Algorithm class} & \multicolumn{2}{c|}{\textbf{Deterministic}} & \multicolumn{2}{c}{\textbf{Randomized}} \\ \hline
\multirow{2}{*}{$\boldsymbol{1^\textnormal{st}}$ \textbf{oracle}} & \textbf{U:} & ---  & \textbf{U:} & $\tilde{O}(\sqrt{d}\delta^{-1}\epsilon^{-4})$ (Remark \ref{rem:randomized1stOracle}) \\
& \textbf{L:} & $+\infty$ (Theorem \ref{Theorem:1LB}) & \textbf{L:} & $\Omega(\epsilon^{-2})$~\citep{Carmon-2020-Lower}\tablefootnote{This lower bound has been proved for finding an $\epsilon$-stationary point of a smooth and nonconvex function. In smooth optimization,~\citet[Proposition~6]{Zhang-2020-Complexity} shows that the notions of Goldstein stationarity is equivalent to the standard notion of stationarity. Thus, the lower bound is valid for a $(\delta, \epsilon)$-Goldstein stationary in smooth optimization and thus transfers to our setting.} \\ \hline
\multirow{3}{*}{$\boldsymbol{0^\textnormal{st} ~ \& ~  1^\textnormal{st}}$ \textbf{oracles}\tablefootnote{We only consider the high-dimensional case with $d \gg 1$. In one-dimensional case,~\citet{Chewi-2022-Complexity} has proved that the complexity bound for finding an $\epsilon$-stationary point of a smooth and nonconvex function is $O(\log(1/\epsilon))$.}} & \textbf{U:} & $(\delta^{-1} \epsilon^{-1})^{O(d)}$\tablefootnote{This can be achieved by a brute force search over the space where there is certainly a solution using classical techniques.} & \textbf{U:} & $\tilde{O}(\delta^{-1}\epsilon^{-3})$~\citep{Davis-2022-Gradient} \\
& \textbf{U:} & $\tilde{O}(M\delta^{-1}\epsilon^{-3})$\tablefootnote{This algorithm needs the extra assumption that the smoothness parameter is $\Theta(2^M)$ rather than $+\infty$.} (Theorem \ref{Theorem:SMOOTH}) & \textbf{U:} & $\tilde{O}(\delta^{-1}\epsilon^{-3})$~\citep{Tian-2022-Finite} \\
& \textbf{L:} & $\Omega(d)$ (Theorem \ref{Theorem:DET}) & \textbf{L:} & $\Omega(\epsilon^{-2})$~\citep{Carmon-2020-Lower} \\ \hline
\end{tabular}
\end{table}
% Concurrently appearing on ArXiv,~\citet{Kornowski-2022-Complexity} have established the stronger dimension-dependent lower bound for all deterministic algorithms that are accessible to first-order and zeroth-order oracles and also presented a deterministic algorithm with a logarithmic dependence on smoothness parameter. Moreover, they derived the lower bounds of $\Omega(\epsilon^{-2})$ and $\Omega(\log(1/\delta))$ for any (possibly randomized) algorithms that are accessible to first-order and zeroth-order oracles. Our results were derived independently and the differences lie in two aspects: (i) the hard instances we use for proving the dimension-dependent lower bound are different, and (ii)~\citet{Kornowski-2022-Complexity} does not identify the distinctive role of zeroth-order oracle for deriving the finite-time convergence in nonsmooth and nonconvex optimization under Lipschitz condition. 

\section{Related Work}\label{sec:related_work}
To appreciate the difficulty and the scope of research agenda in nonsmooth nonconvex optimization, we start by describing the relevant literature. In this context, the existing research are mostly devoted to establishing the asymptotic convergence of optimization algorithms, including the gradient sampling (GS) method~\citep{Burke-2002-Approximating, Burke-2002-Two, Burke-2005-Robust, Kiwiel-2007-Convergence, Burke-2020-Gradient}, bundle methods~\citep{Kiwiel-1996-Restricted} and subgradient methods~\citep{Benaim-2005-Stochastic, Davis-2020-Stochastic, Daniilidis-2020-Pathological, Bolte-2021-Conservative}. More specifically,~\citet{Burke-2002-Approximating} provided the systematic investigation of approximating a generalized gradient through a simple yet novel random sampling scheme, motivating the subsequent development of celebrated gradient bundle method~\citep{Burke-2002-Two}. Then,~\citet{Burke-2005-Robust} and~\citet{Kiwiel-2007-Convergence} proposed the modern GS method by incorporating key modifications into the scheme of the aforementioned gradient bundle method and proved that any cluster point of the iterates generated by the GS method is a Clarke stationary point. For an overview of GS methods, we refer to~\citet{Burke-2020-Gradient}. It is worth mentioning that~\citet{Kiwiel-1996-Restricted} generalized the bundle method to nonsmooth nonconvex optimization by using an affine model that embeds downward shifts. 

There has been recent progress in the investigation of different subgradient methods for nonsmooth nonconvex optimization. It was shown by~\citet{Daniilidis-2020-Pathological} that the standard subgradient method fails to find any Clarke stationary point of a Lipschitz function, as witnessed by the existence of pathological examples.~\citet{Benaim-2005-Stochastic} established the asymptotic convergence guarantee of stochastic approximation methods from a differential inclusion point of view under additional conditions and~\citet{Bolte-2021-Conservative} justified \textit{automatic differentiation} as used in deep learning.~\citet{Davis-2020-Stochastic} proved the asymptotic convergence of subgradient methods if the objective function is assumed to be Whitney stratifiable. Turning to nonasymptotic convergence guarantee,~\citet{Zhang-2020-Complexity} proposed a randomized variant of Goldstein's subgradient method and proved a dimension-independent complexity bound of $O(\delta^{-1}\epsilon^{-3})$ for finding a $(\delta, \epsilon)$-Goldstein stationary point of a Hadamard directionally differentiable function. For a more broad class of Lipschitz functions,~\citet{Davis-2022-Gradient} and~\citet{Tian-2022-Finite} have proposed two other randomized variants of Goldstein's subgradient method and proved the same complexity guarantee. Comparing to their randomized counterparts, the deterministic algorithms are relatively scarce in nonsmooth nonconvex optimization. 

In contrast to the algorithmic progress, the study of algorithm-independent lower bounds for finding any stationary point in nonsmooth nonconvex optimization remains unexplored. In convex optimization, we have a deep understanding of the complexity of finding an $\epsilon$-optimal point (i.e., $\x \in \br^d$ satisfying that $f(\x) - \min_{\x \in \br^d} f(\x) \leq \epsilon$)~\citep{Nemirovski-1983-Problem, Guzman-2015-Lower, Braun-2017-Lower, Nesterov-2018-Lectures}. In smooth nonconvex optimization, various lower bounds have been established for finding an $\epsilon$-stationary point (i.e., $\x \in \br^d$ satisfying that $\|\nabla f(\x)\| \leq \epsilon$)~\citep{Vavasis-1993-Black, Nesterov-2012-Make, Carmon-2020-Lower, Carmon-2021-Lower}. Further extensions to nonconvex stochastic optimization were given in~\citet{Arjevani-2020-Second, Arjevani-2022-Lower} and the algorithm-specific lower bounds for finding an $\epsilon$-stationary point were derived in~\citet{Cartis-2010-Complexity, Cartis-2012-Complexity, Cartis-2018-Worst}. However, these proof techniques can not be extended to nonsmooth nonconvex optimization due to different optimality notions. In this vein,~\citet{Zhang-2020-Complexity} and~\citet{Kornowski-2021-Oracle} have demonstrated that neither an $\epsilon$-Clarke stationary point nor a near $\epsilon$-Clarke stationary point can be obtained in a $\textnormal{poly}(d, \epsilon^{-1})$ number of queries when $\epsilon > 0$ is smaller than some constant. Our analysis is inspired by their construction and techniques but focus on establishing lower bounds for finding a $(\delta, \epsilon)$-Goldstein stationary point.   

The smoothing viewpoint starts with~\citet[Theorem~9.7]{Rockafellar-2009-Variational}, which states that any approximate Clarke stationary point of a Lipschitz function is the asymptotic limit of appropriate approximate stationary points of smooth functions. In particular, given a Lipschitz function $f$, we try to construct a smooth function $\tilde{f}$ that is $\delta$-close to $f$ (i.e., $\|f-g\|_\infty \leq \delta$), and apply a smooth optimization algorithm on $\tilde{f}$. Such smoothing approaches have been used in convex optimization~\citep{Nesterov-2005-Smooth, Beck-2012-Smoothing} and found the application in structured nonconvex optimization~\citep{Chen-2012-Smoothing}. For a general Lipschitz function,~\citet{Duchi-2012-Randomized} proposed a randomized smoothing approach that can transform the original problem to a smooth nonconvex optimization where the objective function is given in the expectation form and the smoothness parameter has to be dimension-dependent. Moreover, we are aware of deterministic smoothing approaches that yield dimension-independent smoothness parameters but emphasize that they are computationally intractable~\citep{Lasry-1986-Remark, Attouch-1993-Approximation}. Recently,~\citet{Kornowski-2021-Oracle} have explored the trade-off between computational tractability and smoothing, ruling out the existence of any (possibly randomized) smoothing approach that achieves computational tractability and dimension-independent smoothness parameter. However, their analysis does not give any concrete deterministic smoothing approach under a Lipschitz condition. 

\section{Preliminaries and Technical Background}\label{sec:prelim}
We start by presenting our notation and providing the formal definitions for the class of functions and the optimality condition (a relaxation of standard stationarity) considered in this paper. We also set up the algorithm class and the notions of complexity measures which will be used to derive a lower bound for any deterministic algorithm in nonsmooth nonconvex optimization. 

\paragraph{Notation.} We let $[d] := \{1, 2, \ldots, d\}$ denote the set of positive integers less than or equal to $d \geq 1$ and define $\br^d$ as the set of $d$-dimensional vectors in the Euclidean space. In particular, we let $\one_d \in \br^d$ and $\zero_d \in \br^d$ be the vectors whose entries are all ones and zeros, respectively. We let $e_1, e_2, \ldots, e_d$ denote a sequence of standard basis vectors and let $I_d \in \br^{d \times d}$ be the $d \times d$ identity matrix. For a vector $\x \in \br^d$, its Euclidean norm refers to $\|\x\|$ and its $i^\textnormal{th}$ coordinate refers to $x_i$. We denote $\supp\{\x\} = \{i \in [d]: x_i \neq 0\}$ as the support of $\x$ (i.e., a collection of nonzero indices).  For a set $\XCal \subseteq \br^d$, we let $\textnormal{conv}(\XCal)$ denote its convex hull. For a continuous function $f(\cdot): \br^d \mapsto \br$, we let $\nabla f(\x)$ denote the gradient of $f$ at $\x$ (if it exists at this point). For a scalar $a \in \br$, we let $\lfloor a \rfloor$ and $\lceil a\rceil$ be the smallest integer that is larger than $a$ and the largest integer that is smaller than $a$. In addition, we define $\rho$-neighborhood of a point $\x \in \br^d$ by $\BB_R(\x) := \{\y \in \br^d: \|\y - \x\| \leq R\}$. Finally, we use the standard notation, with $O(\cdot)$, $\Theta(\cdot)$ and $\Omega(\cdot)$ to hide the absolute constants that do not depend on problem parameters, and $\tilde{O}(\cdot)$, $\tilde{\Theta}(\cdot)$ and $\tilde{\Omega}(\cdot)$ to hide the absolute constants and additional logarithmic factors.

\subsection{Function classes}
Imposing the reasonable regularity conditions to the objective functions is pivotal for the development of a meaningful complexity theory for various optimization algorithms~\citep{Nemirovski-1983-Problem}. A minimal set of conditions that have become standard for nonsmooth and nonconvex optimization are Lipschitzian properties of function values and bounds on function values.  

We first review the definition of Lipschitzian properties of function values. A function $f: \br^d \mapsto \br$ is said to be $L$-Lipschitz if for every $\x \in \br^d$ and the direction $\sv \in \br^d$ with $\|\sv\| \leq 1$, the directional projection $f_{\x, \sv}(t) := f(\x + t\sv)$ satisfies 
\begin{equation*}
|f_{\x, \sv}(t_1) - f_{\x, \sv}(t_2)| \leq L|t_1 - t_2|,  \quad \textnormal{for all } t_1, t_2 \in \br. 
\end{equation*}
Equivalently,  the function $f$ is said to be $L$-Lipschitz if the following statement holds true, 
\begin{equation*}
|f(\x_1) - f(\x_2)| \leq L\|\x_1 - \x_2\|, \quad \textnormal{for all } \x_1, \x_2 \in \br^d. 
\end{equation*}
The key quantity $f(\x^0) - \inf_{\x \in \br^d} f(\x)$ appears in the complexity bound for optimization algorithms in both convex and nonconvex optimization~\citep{Nesterov-2018-Lectures}, where $\x^0 \in \br^d$ is a an initial point for the algorithm. It is often assumed that $f(\x^0) - \inf_{\x \in \br^d} f(\x) \leq \Delta$ where $\Delta > 0$ is a dimension-independent constant. Fixing $\x^0 = \zero_d$ (without loss of generality), we consider the following classes of functions throughout this paper. 
\begin{definition}\label{def:function_class}
Suppose that the problem parameters $\Delta, L > 0$ are independent of the dimension $d \geq 1$. Then, we let $\FCal_d(\Delta, L)$ denote the set of $L$-Lipschitz functions $f: \br^d \mapsto \br$ with bounded function value gap; that is, $f(\zero_d) - \inf_{\x \in \br^d} f(\x) \leq \Delta$. 
\end{definition}
We see from Definition~\ref{def:function_class} that $\FCal_d(\Delta, L)$ is defined in a different manner from~\citet[Definition~1]{Carmon-2020-Lower}. In particular, we do not impose any smoothness condition on the functions $f \in \FCal_d(\Delta, L)$, in contrast to the functions considered in~\citet[Definition~1]{Carmon-2020-Lower} which are assumed to be infinitely differentiable. Notably, our class is defined for a \textit{fixed and finite} $d$ (but can be very large). In contrast, the function class in~\cite[Definition~1]{Carmon-2020-Lower} includes smooth functions on $\br^d$ for any dimension $d \geq 1$. Indeed, their construction followed the established \textit{dimension-independent} complexity guarantee in convex optimization~\citep{Nemirovski-1983-Problem}, demonstrating that the lower bound for finding an $\epsilon$-stationary point is independent of the dimension $d \geq 1$ in smooth and nonconvex optimization if $d$ is sufficiently large such that $d = \Omega(\epsilon^{-1})$.  

The goal of our paper is to show that the dimension-independent complexity guarantee never holds true in \textit{deterministic} nonsmooth and nonconvex optimization by providing a \textit{dimension-dependent} lower bound regardless of the relationship between $d$ and $\epsilon$. In fact, we intend to prove that any deterministic algorithm requires at least $\Omega(d)$ number of queries for finding a $(\delta, \epsilon)$-Goldstein stationary point (see Definition~\ref{def:Goldstein}) when $\delta, \epsilon > 0$ are smaller than \textit{some constant}. As such, it suffices to consider a function class with a fixed and finite dimension $d \geq 1$. The proof is simple yet nontrivial: for any deterministic algorithm, we can construct a hard function $f: \br^d \mapsto \br$ such that the required number of queries is at least $\Omega(d)$ when $\delta, \epsilon > 0$ are smaller than some constant (see Section~\ref{sec:lower} for the details). 

\subsection{Generalized gradients and stationary points}
We start with the definition of generalized gradients~\citep{Clarke-1990-Optimization} for nondifferentiable functions. This is perhaps the most natural extension of gradients to nonsmooth and nonconvex functions. 
\begin{definition}
Given a point $\x \in \br^d$ and a direction $\sv \in \br^d$, the generalized directional derivative of a nondifferentiable function $f$ is given by 
\begin{equation*}
D f(\x; \sv) := \limsup_{\y \rightarrow \x, t \downarrow 0} \tfrac{f(\y + t\sv) - f(\y)}{t}. 
\end{equation*}
The generalized gradient of $f$ is defined as $\partial f(\x) := \{\g \in \br^d: \g^\top \sv \leq D f(\x; \sv) \textnormal{ for all } \sv \in \br^d\}$. 
\end{definition} 
We shall see that there are several equivalent ways of defining a generalized gradient of a Lipschitz function without the use of generalized directional derivatives. One alternate hinges upon Rademacher's Theorem~\citep[Section 3.1.2]{Evans-2018-Measure}, which asserts that the gradient of a Lipschitz function exists almost everywhere, and provides a simple and intuitive characterization: $\partial f(\x)$ is the convex hull of limit points of $\nabla f(\x_k)$ over all sequences $\{\x_k\}_{k \geq 1}$ of differentiable points of $f(\cdot)$ that converge to $\x$. Formally, we summarize some basic properties of generalized gradient of a Lipschitz function and refer the interested readers to~\citet[Section~2]{Clarke-1990-Optimization} for proof details. 
\begin{proposition}
Suppose that a function $f$ is $L$-Lipschitz for some $L > 0$, we have 
\begin{equation*}
\partial f(\x) := \textnormal{conv}\left\{\g \in \br^d: \g = \lim_{\x_k \rightarrow \x} \nabla f(\x_k)\right\}. 
\end{equation*}
Moreover, we have: (i) $\partial f(\x)$ is an nonempty, convex and compact set with $\|\g\| \leq L$ for all $\g \in \partial f(\x)$; (ii) $\partial f(\cdot)$ is an upper-semicontinuous set valued map; and (iii) the mean-value theorem holds: for any $\x_1, \x_2 \in \br^d$, there exists $\lambda \in (0, 1)$ and $\g \in \partial f(\lambda\x_1 + (1-\lambda)\x_2)$ such that $f(\x_1) - f(\x_2) = \g^\top(\x_1 - \x_2)$. 
\end{proposition}
With the notion of generalized gradient, we say $\x \in \br^d$ is a Clarke stationary point of $f$ if it satisfies that $\zero_d \in \partial f(\x)$. Accordingly, a point $\x \in \br^d$ is said to be an $\epsilon$-Clarke stationary if $\min \{\|\g\|: \g \in \partial f(\x)\} \leq \epsilon$. Then, it is natural to ask if one can derive the lower bound for finding an $\epsilon$-stationary point as an analog of~\citet{Carmon-2020-Lower}. In this context, this question was formally addressed in~\citet[Theorem~1]{Zhang-2020-Complexity} that finding an $\epsilon$-Clarke stationary point in nonsmooth and nonconvex optimization can not be achieved in finite time given a fixed tolerance $\epsilon \in [0, 1)$. 

A straightforward relaxation of $\epsilon$-Clarke stationarity is near $\epsilon$-Clarke stationarity. In particular, we consider a point that is $\delta$-close to an $\epsilon$-stationary point for some $\delta > 0$; that is, a point $\x \in \br^d$ is near $\epsilon$-stationary if $\min\{\|\g\|: \g \in \cup_{\y \in \BB_\delta(\x)}  \partial f(\y)\} \leq \epsilon$. However, the recent result of~\cite[Theorem~1]{Kornowski-2021-Oracle} ruled out its computationally tractability by showing that the required number of queries for finding a near $\epsilon$-Clarke stationary point of $f \in \FCal_d(\Delta, L)$ has an exponential dependence on the dimension $d \geq 1$ when $\epsilon, \delta > 0$ are smaller than some constants. 

These negative results suggest the need for further relaxing the targeted stationarity while keeping the close relationship between the relaxed one and Clarke stationarity. Currently, most of the approaches are contingent on the following celebrated notion of Goldstein subdifferential~\citep{Goldstein-1977-Optimization}.
\begin{definition}\label{def:Goldstein}
Given a point $\x \in \br^d$ and $\delta > 0$, the $\delta$-Goldstein subdifferential of a Lipschitz function $f$ at $\x$ is defined as $\partial_\delta f(\x) := \textnormal{conv}(\cup_{\y \in \BB_\delta(\x)}  \partial f(\y))$. 
\end{definition}
The Goldstein subdifferential of $f$ at $\x$ stands for a convex hull of the union of generalized gradients at each point in a $\delta$-neighborhood of $\x$. Then, we define the $(\delta, \epsilon)$-Goldstein stationary points properly; that is, a point $\x \in \br^d$ is a $(\delta, \epsilon)$-Goldstein stationary point if 
\begin{equation*}
\min \left\{\|\g\|: \g \in \partial_\delta f(\x) \right\} \leq \epsilon. 
\end{equation*}
The $(\delta, \epsilon)$-Goldstein stationarity is generally weaker than $\epsilon$-Clarke stationarity since any $\epsilon$-Clarke stationary point is a $(\delta, \epsilon)$-Goldstein stationary point but not vice versa. However, these two notions are equivalent if $f$ is assumed to be smooth~\cite[Proposition~6]{Zhang-2020-Complexity}. Moreover,~\citet[Lemma~7]{Zhang-2020-Complexity} shows that $\lim_{\delta \downarrow 0} \partial_\delta f(\x) = \partial f(\x)$, enabling an feasible framework for transforming nonasymptotic results for finding a $(\delta, \epsilon)$-Goldstein stationary point (if exists) to asymptotic results for finding a Clark stationary point in the literature~\citep{Burke-2020-Gradient}.  Therefore, the $(\delta, \epsilon)$-Goldstein stationarity serves as a reasonable optimality criterion for deriving the finite-time convergence guarantee of the algorithms in nonsmooth and nonconvex optimization. 
\begin{remark}
Finding a $(\delta, \epsilon)$-Goldstein stationary point in nonsmooth nonconvex optimization is computationally tractable when $f$ is Lipschitz and the nonasymptotic analysis has been already done for randomized algorithms~\citep{Davis-2022-Gradient, Tian-2022-Finite}. In particular, they proposed the randomized variants of Goldstein's subgradient method and proved the complexity bound of $\tilde{O}(\delta^{-1}\epsilon^{-3})$.  Nonetheless, it still remains unknown what lower and upper bounds for finding a $(\delta, \epsilon)$-Goldstein stationary point in deterministic nonsmooth and nonconvex optimization are. 
\end{remark}

\subsection{Algorithm class and complexity measures}
We present the proper definition of the class of optimization algorithms considered in this paper.  Since the dimension $d \geq 1$ is fixed and finite, an \textit{algorithm} $\algo$ is defined to map each function $f: \br^d \mapsto \br$ to \textit{the sequence of iterates} in $\br^d$; indeed, we let $\algo[f] = \{\x^t\}_{t \geq 0} \subseteq \br^d$ denote the sequence of iterates that the algorithm $\algo$ generates when operating on $f$. 

The general framework we consider to measure the complexity of finding a $(\delta, \epsilon)$-Goldstein stationary point in nonsmooth nonconvex optimization is the classical information-based oracle model~\citep{Nemirovski-1983-Problem}, where the algorithm $\algo$ has the access to the function $f \in \FCal_d(\Delta, L)$ only by querying a local oracle $\bo_f$ such that 
\begin{equation*}
\x^t = \algo^{(t)}(\x^0, \bo_f(\x^0), \bo_f(\x^1), \ldots, \bo_f(\x^{t-1})), 
\end{equation*}
The local oracle $\bo_f$ means that the information the oracle returns about a function $f_1$ when queried at a point $\x$ is identical to that it returns when a function $f_2$ is queried at $\x$ whenever $f_1(\z) = f_2(\z)$ for all $\z \in \BB_\rho(\x)$ with some $\rho > 0$. A typical example is $\bo_f(\x) = (f(\x), \g)$ where $f(\x)$ is the function value and $\g \in \partial f(\x)$ is chosen as any Clarke subgradient of $f$ at $\x$ without taking the global information of $f$ into account. This requirement of locality allows us to rule out many unnatural situations and is widely accepted in the literature~\citep{Braun-2017-Lower, Kornowski-2021-Oracle}. Any deterministic algorithm that is accessible to $1^\textnormal{st}$ and $0^\textnormal{th}$ oracles sequentially queries the iterates using the local oracle $\bo_f(\x) = (f(\x), \g)$ and use this information to pursue a $(\delta, \epsilon)$-Goldstein stationary point. 

We let $\ACal_{\det}$ be the class of deterministic algorithms that are accessible to $1^\textnormal{st}$ and $0^\textnormal{th}$ oracles and let $\ACal_{\zr}$ be the subclass of $\ACal_{\det}$ where all the algorithms are \textit{zero-respecting}. Note that $\ACal_{\zr}$ is important to proving lower bounds in nonconvex optimization~\citep{Carmon-2020-Lower, Carmon-2021-Lower} since it is not only small enough to perform poorly on a single function uniformly, but large enough to imply lower bounds on the algorithm class $\ACal_{\det}$.  Formally, the algorithm $\algo$ is zero-respecting if for any $f: \br^d \mapsto \br$, the iterate sequence $\algo[f] = \{\x^t\}_{t = 1}^\infty$ satisfies that $\supp\{\x^t\} \subseteq \cup_{s < t} \supp\{\g^s\}$ for each $t \geq 1$, where $\g^s \in \partial f(\x^s)$ is a Clarke subgradient used in the algorithm $\algo$. Notably, the above definition is equivalent to the requirements that (i) $\x^0 = \zero_d$ and (ii) for every $t \geq 1$ and $j \in [d]$, if $g_j^s = 0$ for $s < t$, then $x_j^t = 0$.  Informally speaking, an algorithm $\algo$ is zero-respecting if it never explores coordinates that appear not to affect the function.

With the above notions in hand, we are ready to formalize the key notion of complexity measures: \textit{what is the best performance that a deterministic algorithm in $\ACal$ can achieve for all the functions in $\FCal$?} A natural performance measure is the number of queries required to find a $(\delta, \epsilon)$-Goldstein stationary point $\x \in \br^d$. Formally, given a deterministic sequence $\{\x^t\}_{t \geq 0}$, we define its \textit{complexity} on $f$ by
\begin{equation*}
\bound_{\delta, \epsilon}(\{\x^t\}_{t \geq 0}, f) := \inf\left\{t \geq 0: \min \left\{\|\g\|: \g \in \partial_\delta f(\x^t) \right\} < \epsilon\right\}. 
\end{equation*}
To measure the performance of an algorithm $\algo$ on a function $f$, we evaluate the iterates that $\algo$ produces from $f$, and with abuse of notation, we define $\bound_{\delta, \epsilon}(\algo, f) := \bound_{\delta, \epsilon}(\algo[f], f)$ as the complexity of $\algo$ on $f$. As such, we can define the complexity of algorithm class $\ACal_{\det}$ on function class $\FCal_d(\Delta, L)$ as 
\begin{equation}\label{def:bound}
\TCal_{\delta, \epsilon}(\ACal_{\det}, \FCal_d(\Delta, L)) := \inf_{\algo \in \ACal_{\det}} \sup_{f \in \FCal_d(\Delta, L)} \bound_{\delta, \epsilon}(\algo, f). 
\end{equation}
Given these definitions, we ask whether or not it is possible to prove a dimension-dependent lower bound of $\Omega(d)$ for $\TCal_{\delta, \epsilon}(\ACal_{\det}, \FCal_d(\Delta, L))$ when $\delta, \epsilon > 0$ are smaller than some universal constants. This is clearly an important question in nonconvex optimization but has still remained open to our knowledge. 

We answer the above question in the affirmative, demonstrating the importance of randomization in obtaining the \textit{dimension-independent} complexity guarantee in terms of Goldstein stationarity. The key step in our proofs is to exhibit a hard function $f$ and bound the quantity of $\inf_{\algo \in \ACal_{\zr}} \bound_{\delta, \epsilon}(\algo, f)$ from below (see Section~\ref{sec:lower} for the details). 

\section{Dimension-Dependent Lower Bound}\label{sec:lower}
In this section, we prove the dimension-dependent lower bounds for deterministic algorithms that are accessible to both $1^\textnormal{st}$ and $0^\textnormal{th}$ oracles in nonsmooth and nonconvex optimization. The proof here is based on the modification of a hard function in~\citet{Kornowski-2022-Complexity} and the application of classical techniques~\citep{Nemirovski-1983-Problem, Carmon-2020-Lower, Carmon-2021-Lower} with $(\delta, \epsilon)$-Goldstein stationarity. 

\subsection{Overview of classical techniques}
As a warm-up, we review the classical techniques for proving lower bounds for zero-respecting algorithms in nonsmooth and convex optimization. Following~\citet[Chapter 3.1.2]{Nesterov-2018-Lectures}, we fix the dimension-independent parameters $R > 0$, $L > 0$ and $\epsilon > 0$ and assume that the dimension $d$ is sufficiently large such that $d \geq \lfloor 10L^2R^2\epsilon^{-2}\rfloor + 1$. For any zero-respecting algorithms with $\x^0 = \zero_d$, our goal is to construct a nonsmooth and convex function $f$ satisfying that (i) $\|\x^0 - \x^\star\| \leq R$ where $\x^\star$ is the unique global minimum; (ii) $f$ is $L$-Lipschitz over $\BB_R(\x^\star)$; and (iii) $f(\x^t) - f(\x^\star) \geq \epsilon$ for all $t \leq \lfloor\frac{L^2R^2\epsilon^{-2}}{36} \rfloor$~\footnote{It is reasonable to assume that $\frac{L^2R^2\epsilon^{-2}}{36} \geq 1$. If this does not hold, we will see later that $f(\zero_d) - \min_{\x \in \br^d} f(\x) = \frac{LR}{6} \leq \epsilon$ for the function $f$ defined in Eq.~\eqref{example:convex}. This means that a trivial $\epsilon$-optimal solution exists and no optimization is needed.}.  

For simplicity, we let $T = \lfloor \frac{L^2R^2\epsilon^{-2}}{36}\rfloor + 1$ and define the single hard function as follows, 
\begin{equation}\label{example:convex}
f(\x) = \tfrac{L}{3} \max_{1 \leq i \leq T} \{x_i\} + \tfrac{L}{6R\sqrt{T}}\|\x\|^2. 
\end{equation} 
Since $d \geq \lfloor 10L^2R^2\epsilon^{-2}\rfloor + 1$, we have $2 \leq T < d$ and the above function in Eq.~\eqref{example:convex} is well defined. Since $f$ is strongly convex, the global minimum of $f$ is unique and we let this point denote $\x^\star$. By definition, we have $\zero_d \in \partial f(\x^\star)$ and $\partial f(\x) = \frac{L}{3R\sqrt{T}} \x + \frac{L}{3} \cdot \textnormal{conv}(\{e_j: j = \argmax_{1 \leq i \leq T} x_i\})$. Then, the following statements hold true, 
\begin{equation*}
f(\x^\star) = -\tfrac{LR}{6\sqrt{T}}, \qquad x_i^\star = \left\{
\begin{array}{cl} -\tfrac{R}{\sqrt{T}}, & \textnormal{if } 1 \leq i \leq T,  \\ 0, & \textnormal{otherwise}.
\end{array}\right. .
\end{equation*}
We are ready to verify (i) and (ii) as desired. Indeed, for (i), we have $\|\x^0 - \x^\star\| = \|\x^\star\| = \sqrt{\sum_{i=1}^T \frac{R^2}{T}} = R$. For (ii), let $\x, \x' \in \BB_R(\x^\star)$, we have $f(\x) - f(\x') \leq \|\xi\|\|\x - \x'\|$ for any $\xi \in \partial f(\x)$. By the definition of $\partial f(\x)$, we have $\|\xi\| \leq \frac{L}{3R\sqrt{T}}\|\x\| + \frac{L}{3} \leq \frac{L}{3R\sqrt{T}}(\|\x^\star\| + R) + \frac{L}{3} \leq L$ for any $\xi \in \partial f(\x)$ and $\x \in \BB_R(\x^\star)$. Thus, $f$ is $L$-Lipschitz over $\BB_R(\x^\star)$.  

It suffices to verify (iii). Since the function $f$ is fixed and the local oracle $\bo_f(\x)$ returns the function value $f(\x)$ and \textit{any} Clarke subgradient $ \g \in \partial f(\x)$ without taking the global information into account, we can let $\bo_f(\x)$ provide the most informative subgradient at a query point; indeed, we receive $\frac{L}{3}e_{i^\star} + \frac{L}{3R\sqrt{T}}\x$ from $\bo_f(\x)$ where $i^\star \leq d$ is the smallest index so that $x_{i^\star} = \max_{1 \leq i \leq T} x_i$. Since $\x^0 = \zero_d$, we have $f(\x^0) = 0$ and $\g^0 = \frac{L}{3} e_1$.  For any zero-respecting algorithm,  we have $\supp\{\x^1\} \subseteq \supp\{\g^0\} = \{1\}$ which implies that $x_i^1 = 0$ for all $2 \leq i \leq d$. As such, we have $\supp\{\g^1\} \subseteq \{1, 2\}$ and $\supp\{\x^2\} \subseteq \{1, 2\}$ which implies that $x_i^2 = 0$ for all $3 \leq i \leq d$.  Repeating this argument, we have $x_i^t = 0$ for all $t+1 \leq i \leq d$; that is, we only \textit{discover} a single new coordinate at each iteration (using a query). Consequently, for all $1 \leq t \leq \lfloor \frac{L^2R^2\epsilon^{-2}}{36} \rfloor < T$, we have $x_T^t = 0$ and thus 
\begin{equation*}
f(\x^t) \geq \tfrac{L}{3} \max_{1 \leq i \leq T} x_i^t \geq 0. 
\end{equation*}
This implies that $f(\x^t) - f(\x^\star) \geq \tfrac{LR}{6\sqrt{T}} \geq \epsilon$. Putting these pieces together yields the desired result. 

\subsection{Lower bound for zero-respecting algorithms}
We proceed to prove a lower bound for finding a $(\delta, \epsilon)$-Goldstein stationary point of $f \in \FCal_d(\Delta, L)$ using a class of zero-respecting algorithms in $\ACal_{\zr}$. In particular, we fix the dimension-independent parameters $\Delta > 0$, $L > 0$, $\delta > 0$ and $\epsilon > 0$ and the dimension $d \geq 2$ is assumed to be both fixed and finite\footnote{The case of $d = 1$ has been studied in~\citet{Zhang-2020-Complexity}. Since our goal is to derive the dependence of lower bounds on $d$, we simply assume that $d \geq 2$ is finite (without loss of generality).}. For any zero-respecting algorithms with $\x^0 = \zero_d$, our goal is to construct a nonsmooth and nonconvex function $f$ satisfying that (i) $f(\x^0) - f^\star \leq \Delta$ where $f^\star = \min_{\x \in \br^d} f(\x)$ is the global optimal function value; (ii) $f$ is $L$-Lipschitz over $\br^d$; and (iii) $\min \{\|\g\|: \g \in \partial_\delta f(\x^t)\} \geq \epsilon$ for all $t \geq 1$\footnote{Our impossibility result here implies that any algorithm in $\ACal_{\zr}$ can not find a $(\delta, \epsilon)$-Goldstein stationary point within the finite number of iterations. The same result has been obtained by~\citet{Tian-2022-No} using different hard instances. }. 

Fixing $d \geq 2$ and letting the dimension-independent parameters be $\Delta > 0$, $L > 0$, $0 < \delta \leq \frac{\Delta}{2L}$ and $0 < \epsilon < \frac{L}{252}$. For any $T \geq 1$, our resisting strategy is that the local oracle $\bo_f(\x^t)$ will return $f(\x^t) = 0$ and $\nabla f(\x^t) = \frac{1}{7}L e_1$ for all $0 \leq t \leq T$. Since the algorithm is zero-respecting and $\x^0 = \zero_d$, the above strategy fixes the iterates $\x^1, \x^2, \ldots, \x^T$ and $\supp\{\x^t\} \subseteq \{1\}$ for all $1 \leq t \leq T$. In the following, we will show that this resisting strategy is indeed consistent with a function $f \in \FCal_d(\Delta, L)$. 

\paragraph{Construction.} Since $T$ is finite, we let $r = \frac{1}{4}\min_{0 \leq i < j \leq T}\|\x^i - \x^j\| > 0$ (without loss of generality). For any $\x \in \br^d$, we define the following component functions given by 
\begin{equation*}
g_{\x^t}(\x) = \min\left\{1, \tfrac{\|\x-\x^t\|^2}{r^2}\right\}e_2^\top\x + \left(1 - \min\left\{1, \tfrac{\|\x-\x^t\|^2}{r^2}\right\}\right)e_1^\top(\x - \x^t), \quad \textnormal{for all } t = 0, 1, \ldots, T.  
\end{equation*}
Then, we further define the hard function as follows, 
\begin{equation}\label{example:nonconvex}
f(\x) = \tfrac{L}{7}\max\{h(\x), -\tfrac{7\Delta}{L}\}, 
\end{equation}
where $h: \br^d \mapsto \br$ is a nonsmooth and nonconvex function given by 
\begin{equation*}
h(\x) = \left\{\begin{array}{cl}
g_{\x^t}(\x), & \textnormal{for any } \x \in \BB_r(\x^t), \\
e_2^\top\x, & \textnormal{otherwise.}
\end{array}\right. 
\end{equation*}
\paragraph{Consistent with resisting strategy.} It is clear that $g(\x^t) = 0$ for all $0 \leq t \leq T$. For any $\x \in \BB_r(\x^t)$, we have 
\begin{equation}\label{def:grad-resist}
\nabla g_{\x^t}(\x) = \tfrac{2e_2^\top\x}{r^2}(\x - \x^t) + \tfrac{\|\x - \x^t\|^2}{r^2}e_2 - \tfrac{2e_1^\top(\x - \x^t)}{r^2}(\x - \x^t) - \tfrac{\|\x - \x^t\|^2}{r^2}e_1 + e_1. 
\end{equation}
Thus, we have $\nabla g_{\x^t}(\x^t) = e_1$ for all $0 \leq i \leq T$. By appealing to the definition of $f$ and $h$, we have $f(\x^t) = \frac{1}{7}Lg(\x^t)$ and $\nabla f(\x^t) = \frac{1}{7}L\nabla g_{\x^t}(\x^t)$ for all $0 \leq t \leq T$. This implies the desired result. 

\paragraph{Main analysis.} We are ready to prove (i) and (ii). For (i), we see from the definition of $h(\x)$ that 
\begin{equation*}
f(\x^0) = \tfrac{L}{7}\max\{h(\x^0), -\tfrac{7\Delta}{L}\} = \tfrac{L}{7}\max\{h(\zero_d), -\tfrac{7\Delta}{L}\} = 0, 
\end{equation*}
and $f^\star = \min_{\x \in \br^d} f(\x) = \tfrac{L}{7} \cdot (-\tfrac{7\Delta}{L}) = -\Delta$. Putting these pieces together yields that $f(\x^0) - f^\star \leq \Delta$. Thus, (i) is satisfied. For (ii), we first prove that $h$ is continuous. It suffices to verify the points lying on the boundary of each $\BB_r(\x^t)$. Indeed, we consider $\bar{\x}$ satisfying that $\|\bar{\x} - \x^t\| = r$ and have 
\begin{eqnarray*}
\lefteqn{\lim_{\x \rightarrow \bar{\x}, \x \in \BB_r(\x^t)} h(\x) = \lim_{\x \rightarrow \bar{\x}, \x \in \BB_r(\x^t)} g_{\x^t}(\x)} \\
& = & \lim_{\x \rightarrow \bar{\x}, \x \in \BB_r(\x^t)} \min\left\{1, \tfrac{\|\x-\x^t\|^2}{r^2}\right\}e_2^\top\x + \left(1 - \min\left\{1, \tfrac{\|\x-\x^t\|^2}{r^2}\right\}\right)e_1^\top(\x - \x^t) \ = \ e_2^\top\bar{\x}. 
\end{eqnarray*}
Since $\x \mapsto e_2^\top \x$ is clearly $1$-Lipschitz over $\br^d$, it suffices to prove that $g_{\x^t}(\cdot)$ is 7-Lipschitz over $\BB_r(\x^t)$. Since $\x^0 = \zero_d$ and $\supp\{\x^t\} \subseteq \{1\}$ for all $1 \leq t \leq T$, we have $e_2^\top \x^t = 0$. Then, Eq.~\eqref{def:grad-resist} implies
\begin{eqnarray*}
\|\nabla g_{\x^t}(\x)\| & = & \left\|\tfrac{2e_2^\top(\x - \x^t)}{r^2}(\x - \x^t) + \tfrac{\|\x - \x^t\|^2}{r^2}e_2 - \tfrac{2e_1^\top(\x - \x^t)}{r^2}(\x - \x^t) - \tfrac{\|\x - \x^t\|^2}{r^2}e_1 + e_1\right\| \\
& \leq & \tfrac{4\|\x - \x^t\|^2}{r^2} + \tfrac{2\|\x - \x^t\|^2}{r^2} + \|e_1\| \ \leq \ 7. 
\end{eqnarray*}
Putting these pieces together yields that $h$ is $7$-Lipschitz over $\br^d$. Thus, Eq.~\eqref{example:nonconvex} guarantees that $f$ is $L$-Lipschitz over $\br^d$ and this implies that (ii) is satisfied. 

It suffices to verify (iii). The key step is to show that $h$ has no $(\delta, \frac{1}{36})$-Goldstein stationary point for any given $\delta > 0$. Indeed, by the definition of $h$, we have
\begin{equation*}
\nabla h(\x) = \left\{\begin{array}{cl}
\nabla g_{\x^t}(\x), & \textnormal{for any } \x \textnormal{ such that } \|\x - \x^t\| < r, \\
e_2, &  \textnormal{for any } \x \textnormal{ such that } \|\x - \x^t\| > r \textnormal{ for any } 0 \leq t \leq T.  
\end{array}\right. 
\end{equation*}
For any point $\x \in \br^d$ satisfying $\|\x - \x^t\| = r$, we have $h$ is nondifferentiable and $\partial h(\x)$ consists of the convex combination of $\nabla g_{\x^t}(\x)$ for some $\|\x - \x^t\| < r$ and $e_2$. In addition, we can derive from Eq.~\eqref{def:grad-resist} that the set $\{\nabla g_{\x^t}(\x): \|\x - \x^t\| < r\}$ depends on $\x$ and $\x^t$ only through $\x - \x^t$. This implies that 
\begin{equation*}
\{\nabla g_{\x^t}(\x): \|\x - \x^t\| < r\} = \{\nabla g_{\zero_d}(\z): \|\z\| < r\}. 
\end{equation*}
Since the $\delta$-Goldstein subdifferential at any point is the set of the convex combination of subgradients, and each subgradient can be expressed as a convex combination of gradients at differentiable points, we can simply consider the convex combination of gradients at differentiable points. 

We let $\x \in \br^d$ and $\delta > 0$. Then, for any $\xi \in \partial_\delta h(\x) = \textnormal{conv}(\cup_{\y \in \BB_\delta(\x)}  \partial h(\y)) \subseteq \br^d$, the Carath\'{e}odory's theorem~\citep{Eckhoff-1993-Helly} guarantees that $\xi = \sum_{i=1}^{d+1} \lambda_i\xi_i$ where $\xi_i$ is either $e_2$ or in $\{\nabla g_{\zero_d}(\z): \|\z\| < r\}$ and $\sum_{i=1}^{d+1} \lambda_i = 1$ with $\lambda_i \geq 0$. By abuse of notation, there exists $N \leq d+1$ such that 
\begin{equation*}
\xi = \lambda e_2 + \sum_{i=1}^N \lambda_i\nabla g_{\zero_d}(\z_i), \textnormal{ for } \z_i \in \BB_r(\zero_d) \textnormal{ and } \sum_{i=1}^N \lambda_i = 1 - \lambda \textnormal{ with } \lambda_i \geq 0. 
\end{equation*}
That is to say, for $\z_i \in \BB_r(\zero_d)$ and $\sum_{i=1}^N \lambda_i = 1 - \lambda$ with $\lambda_i \geq 0$, we have
\begin{equation*}
\xi = \lambda e_2 + \sum_{i=1}^N \lambda_i\left(\tfrac{2e_2^\top\z_i}{r^2}\z_i + \tfrac{\|\z_i\|^2}{r^2}e_2 - \tfrac{2e_1^\top\z_i}{r^2}\z_i - \tfrac{\|\z_i\|^2}{r^2}e_1 + e_1\right). 
\end{equation*}
By using the change of variable $\z_i \mapsto \frac{\z_i}{r}$ and $\sum_{i=1}^N \lambda_i = 1 - \lambda$, we have
\begin{eqnarray*}
\lefteqn{\xi = \lambda e_2 + \sum_{i=1}^N \lambda_i\left(2(e_2^\top\z_i)\z_i + \|\z_i\|^2 e_2 - 2(e_1^\top\z_i)\z_i - \|\z_i\|^2e_1 + e_1\right)} \\
& = & \left(\lambda + \sum_{i=1}^N \lambda_i\|\z_i\|^2\right)e_2 + 2\left(\sum_{i=1}^N \lambda_i((e_2 - e_1)^\top\z_i) \z_i\right) + \left(1 - \lambda - \sum_{i=1}^N \lambda_i\|\z_i\|^2\right)e_1. 
\end{eqnarray*}
If $\|\xi\| \geq 1$, we conclude the desired result since $\x \in \br^d$ and $\delta > 0$ are chosen arbitrarily. Otherwise, we assume that $\|\xi\| < 1$ and have 
\begin{eqnarray*}
\lefteqn{\xi^\top e_2 = \lambda + \sum_{i=1}^N \lambda_i\|\z_i\|^2 + 2\left(\sum_{i=1}^N \lambda_i((e_2 - e_1)^\top\z_i) \cdot e_2^\top \z_i\right)} \\
& \geq & \lambda + \sum_{i=1}^N \lambda_i(e_1^\top \z_i)^2 + \sum_{i=1}^N \lambda_i(e_2^\top \z_i)^2 + \sum_{i=1}^N 2\lambda_i(e_2^\top\z_i)^2 - \sum_{i=1}^N 2\lambda_i(e_1^\top\z_i)(e_2^\top \z_i) \\ 
& \geq & \sum_{i=1}^N \lambda_i(e_1^\top \z_i)^2 + \sum_{i=1}^N \lambda_i(e_2^\top \z_i)^2 - \sum_{i=1}^N 2\lambda_i(e_1^\top\z_i)(e_2^\top \z_i) \\
& = & \sum_{i=1}^N \lambda_i((e_2 - e_1)^\top\z_i)^2. 
\end{eqnarray*}
Since $\lambda_i \geq 0$, we have $\xi^\top e_2 \geq 0$. Then, we have
\begin{equation}\label{inequality:key-est}
\xi^\top e_2 \geq (1-\lambda)\xi^\top e_2 \geq \left(\sum_{i=1}^N \lambda_i\right)\left(\sum_{i=1}^N \lambda_i((e_2 - e_1)^\top\z_i)^2\right) \geq \left(\sum_{i=1}^N \lambda_i|(e_2 - e_1)^\top\z_i|\right)^2. 
\end{equation}
Putting these pieces together yields that 
\begin{eqnarray*}
\lefteqn{\xi^\top(e_2 + e_1) = 1 + 2\left(\sum_{i=1}^N \lambda_i((e_2 - e_1)^\top\z_i)((e_2 + e_1)^\top\z_i)\right)} \\
& \overset{\|\z_i\| \leq 1}{\geq} & 1 - 4\left(\sum_{i=1}^N \lambda_i|(e_2 - e_1)^\top\z_i|\right) \ \overset{\textnormal{Eq.~\eqref{inequality:key-est}}}{\geq} \ 1 - 4\sqrt{\xi^\top e_2}. 
\end{eqnarray*}
Combining the above inequality with $\|\xi\| < 1$ yields that 
\begin{equation*}
1 \leq \xi^\top(e_2 + e_1) + 4\sqrt{\xi^\top e_2} \leq \sqrt{2}\|\xi\| + 4\sqrt{\|\xi\|} \leq (4+\sqrt{2})\sqrt{\|\xi\|}. 
\end{equation*}
This implies that $\|\xi\| \geq \frac{1}{36}$. Since $\x \in \br^d$ and $\delta > 0$ are chosen arbitrarily, we obtain the desired result that $h$ has no $(\delta, \frac{1}{36})$-Goldstein stationary point for any given $\delta > 0$. 

Turing back to the function $f$ defined in Eq.~\eqref{example:nonconvex}, we have $f(\x^t) = \frac{L}{7}h(\x^t) = 0$ for all $0 \leq t \leq L$. Since $h$ is $7$-Lipschitz over $\br^d$, we have $h(\x) > -\frac{7\Delta}{L}$ if $\|\x - \x^t\| \leq \frac{\Delta}{2L}$ for some $t$. This further implies that $\partial_\delta f(\x^t) = \frac{L}{7}\partial_\delta h(\x^t)$ if $0 < \delta \leq \frac{\Delta}{2L}$. Since $h$ has no $(\delta, \frac{1}{36})$-Goldstein stationary point for any given $\delta > 0$, we have $\|\xi\| \geq \frac{L}{252}$ for any $\xi \in \partial_\delta f(\x^t)$. This shows that $\x^t$ is not a $(\delta, \epsilon)$-Goldstein stationary point of $f$ for all $0 \leq t \leq T$ if $0 < \delta \leq \frac{\Delta}{2L}$ and $0 < \epsilon < \frac{L}{252}$. 

\paragraph{Conclusion.} Therefore, we conclude that any zero-respecting algorithm cannot return a $(\delta, \epsilon)$-Goldstein stationary point of $f$ defined in Eq.~\eqref{example:nonconvex} if the number of queries is no more than any finite $T$. This further implies that $\TCal_{\delta, \epsilon}(\ACal_{\zr}, \{f\}) \geq T$ for any $T \geq 1$. Based on the above arguments, we are ready to summarize our main results in the following theorem. 
\begin{theorem}\label{Theorem:ZR}
Suppose that $d \geq 2$ is fixed and finite and let $\Delta, L > 0$ be given and independent of $d$.  If $f$ is defined in Eq.~\eqref{example:nonconvex} with $0 < \delta < \frac{\Delta}{2L}$ and $0 < \epsilon \leq \frac{L}{252}$, we have $\TCal_{\delta, \epsilon}(\ACal_{\zr}, \{f\}) \geq T$ for any $T \geq 1$. 
\end{theorem}
\begin{remark}
Theorem~\ref{Theorem:ZR} shows that the finite number of queries is not enough to find a $(\delta, \epsilon)$-Goldstein stationary point of $f$ for any zero-respecting algorithm if $\delta, \epsilon > 0$ are smaller than some small constants. Notably, our results hold true regardless of the relationship between the dimension $d \geq 2$ and parameters $(\delta, \epsilon)$. This is different from the dimension-independent lower bound established for nonsmooth convex optimization~\citep{Nesterov-2018-Lectures}. Our results highlight the importance of convexity for obtaining dimension-independent complexity guarantee in nonsmooth optimization. 
\end{remark}

\subsection{From deterministic to zero-respecting algorithms}
We turn to stating a lower bound for finding $(\delta, \epsilon)$-Goldstein stationary points of Lipschitz functions using the local oracle $\bo_f(\x) = (f(\x), \g)$ and a class of deterministic algorithms (i.e., the class $\ACal_{\det}$).  The following theorem summarizes our results. 
\begin{theorem}\label{Theorem:DET}
Suppose that $d \geq 2$ is fixed and finite and let $\Delta, L > 0$ be given and independent of $d$.  If $0 < \delta < \frac{\Delta}{2L}$ and $0 < \epsilon \leq \frac{L}{252}$, we have $\TCal_{\delta, \epsilon}(\ACal_{\det}, \FCal_{2d}(\Delta, L)) \geq d + 1$. 
\end{theorem}
\begin{remark}
Our lower bound is dimension-dependent and is thus different from dimension-independent lower bounds established in the context of convex optimization~\citep{Nemirovski-1983-Problem, Nesterov-2018-Lectures} and smooth nonconvex optimization~\citep{Carmon-2020-Lower, Carmon-2021-Lower}. Our results highlight that, even though finding a $(\delta, \epsilon)$-Goldstein stationary point in nonsmooth nonconvex optimization is computationally tractable~\citep{Davis-2022-Gradient, Tian-2022-Finite}, it is essentially harder than finding an $\epsilon$-stationary point in smooth nonconvex optimization without randomization. Notably, our lower bound has matched the best existing lower bound established in~\citet{Kornowski-2022-Complexity} and~\citet{Tian-2022-No}. 
\end{remark}
\begin{remark}
Our lower bound can be improved to $\Omega(\max\{d, \frac{\Delta}{L\delta}\})$ via appeal to a simple combination of Theorem~\ref{Theorem:DET} and~\citet[Theorem~11]{Zhang-2020-Complexity}. Indeed, they proved that any algorithm in $\ACal_{\det}$ requires $\Omega(\frac{\Delta}{L\delta})$ number of iterations to find a $(\delta, \epsilon)$-Goldstein stationary point of a function $f \in \FCal_1(\Delta, L)$ when $\epsilon \in (0, L)$. Intuitively, they can construct two different functions in $\FCal_1(\Delta, L)$ so that they share the same gradient norm at all queried points but their stationary points are $\Omega(\delta)$ away if the number of different queried points is less than $\frac{\Delta}{8L\delta}$. Their proof techniques heavily depend on the specific structure of 1-dimensional geometry and are seemingly difficult to be extended to the case of $d \geq 2$. 
\end{remark}
Our proof is based on the classical framework that translates lower bounds from $\ACal_{\zr}$ to $\ACal_{\det}$. Yet, due to different function class in Definition~\ref{def:function_class} (especially the role of $d$) and different optimality criterion in Definition~\ref{def:Goldstein}, we can not apply the results in~\citet{Carmon-2020-Lower} directly but need to reprove some basic properties. The following proposition is the core of the proof of Theorem~\ref{Theorem:DET}.
\begin{proposition}\label{prop:resisting}
Suppose that $d \geq 2$ is fixed and finite and let $\delta, \epsilon > 0$ and $\algo \in \ACal_{\det}$. Then, there exists an algorithm $\textsc{Z}_\algo \in \ACal_{\zr}$ with the following property: for every function $f: \br^d \mapsto \br$, there exists an orthogonal matrix $U \in \br^{2d \times d}$ such that
\begin{equation*}
\bound_{\delta, \epsilon}(\algo, f_U) > d  \textnormal{ or } \bound_{\delta, \epsilon}(\algo, f_U) = \bound_{\delta, \epsilon}(\textsc{Z}_\algo, f), 
\end{equation*}
where $f_U(\x) := f(U^\top \x)$ is the rotated version of the original function $f$. 
\end{proposition} 
\begin{proof}
Following the proof strategy used in~\citet[Appendix A]{Carmon-2020-Lower}, we consider an explicit construction of $\textsc{Z}_\algo \in \ACal_{\zr}$ with the following key property: for every function $f: \br^d \mapsto \br$, there exists an orthogonal matrix $U \in \br^{2d \times d}$ (i.e., $U^\top U = I_d$) such that $f_U(\x) := f(U^\top \x)$ satisfies that the first $d$ iterates in the sequence $\textsc{Z}_\algo[f]$ and $U^\top\algo[f_U]$ are identical. 

We first show that $\textsc{Z}_\algo$ with the aforementioned key property implies that there exists an orthogonal matrix $U \in \br^{2d \times d}$ such that
\begin{equation}\label{result:resisting-first}
\bound_{\delta, \epsilon}(\algo, f_U) \geq d \textnormal{ or } \bound_{\delta, \epsilon}(\algo, f_U) = \bound_{\delta, \epsilon}(\textsc{Z}_\algo, f). 
\end{equation}
Indeed, if $\bound_{\delta, \epsilon}(\algo, f_U) > d$, we are done. Otherwise, let $\{\x^t\}_{t \geq 0}$ be generated by the algorithm $\algo$ on the function $f_U$, we have 
\begin{equation*}
\bound_{\delta, \epsilon}(\algo, f_U) = \bound_{\delta, \epsilon}(\{\x^t\}_{t \geq 0}, f_U) = \inf\left\{t \geq 0: \min \left\{\|\g\|: \g \in \partial_\delta f_U(\x^t) \right\} \leq \epsilon\right\}. 
\end{equation*}
Since $f_U(\x) = f(U^\top \x)$ and $\|U\g\| = \|\g\|$ for all orthogonal matrices $U \in \br^{2d \times d}$, we have 
\begin{equation*}
\min \left\{\|\g\|: \g \in \partial_\delta f_U(\x^t) \right\} = \min \left\{\|\tilde{\g}\|: \tilde{\g} \in \partial_\delta f(U^\top\x^t) \right\}. 
\end{equation*}
Putting these pieces together yields that 
\begin{equation*}
\bound_{\delta, \epsilon}(\algo, f_U) = \inf\left\{t \geq 0: \min \left\{\|\tilde{\g}\|: \tilde{\g} \in \partial_\delta f(U^\top\x^t) \right\} \leq \epsilon\right\} = \bound_{\delta, \epsilon}(\{U^\top\x^t\}_{t \geq 0}, f) = \bound_{\delta, \epsilon}(U^\top\algo[f_U], f). 
\end{equation*}
Since $\bound_{\delta, \epsilon}(\algo, f_U) \leq d$, we have $\bound_{\delta, \epsilon}(U^\top\algo[f_U], f) \leq d$. This implies that the first $d$ iterates of $U^\top\algo[f_U]$ determines $\bound_{\delta, \epsilon}(U^\top\algo[f_U], f)$. Since the first $d$ iterates in the sequence $\textsc{Z}_\algo[f]$ and $U^\top\algo[f_U]$ are identical, we have $\bound_{\delta, \epsilon}(U^\top\algo[f_U], f) = \bound_{\delta, \epsilon}(\textsc{Z}_\algo, f)$. Putting these pieces together yields Eq.~\eqref{result:resisting-first}. 

It remains to construct $\textsc{Z}_\algo \in \ACal_{\zr}$ where the first $d$ iterates can match that of the algorithm $\algo \in \ACal_{\det}$ under an appropriate orthogonal rotation. We conduct this by describing the operation inductively on any the function $f: \br^d \mapsto \br$, which we denote $\{\z^t\}_{t \geq 0} = \textsc{Z}_\algo[f]$. By the definition, the dynamics of the algorithm $\textsc{Z}_\algo$ at the $t^\textnormal{th}$ iteration is determined by a set $S_t \subseteq \{1, 2, \ldots, d\}$ and the orthonormal vectors $\{\su^i\}_{i \in S_t} \subseteq \br^{2d}$ identified with this set. Note that $S_t = \supp\{\g^0, \g^1, \g^2, \ldots, \g^{t-1}\}$. Thus, we have $\emptyset = S_0 = S_1 \subseteq S_2 \subseteq \ldots$ and the collection of vectors $\{\su^i\}_{i \in S_t}$ grows as $t \geq 0$ increases. Then, we can let $U \in \br^{2d \times d}$ be the orthogonal matrix whose $i^\textnormal{th}$ column is $\su^i$ (note that $U$ may not be completely determined throughout the implementation of $\textsc{Z}_\algo$ but we can still simulate the operation of $\algo$ on $f_U$). Letting $\{\x^t\}_{t \geq 0} = \algo[f_U]$, it suffices to show that 
\begin{equation}\label{result:resisting-second}
\z^t = U^\top \x^t \quad \textnormal{and} \quad \supp\{\z^t\} \subseteq S_t, \quad \textnormal{for all } 0 \leq t \leq d - 1. 
\end{equation}
We proceed with the inductive argument. Since $\algo$ is deterministic, the iterate $\x^0 \in \br^{2d}$ is an arbitrary (but deterministic) vector. To Eq.~\eqref{result:resisting-second}, we need to pick up $\{\su^i\}_{1 \leq i \leq d} \in \br^{2d}$ such that $(\su^i)^\top \x^0 = 0$. This implies that the first iterate of $\textsc{Z}_\algo$ satisfies $\z^0 = \zero_d$. Then, we show that $\textsc{Z}_\algo$ can emulate $\x^1$ and from it can construct $\z^1$ that satisfies Eq.~\eqref{result:resisting-second}. To obtain $\x^1$, we require a generalized subgradient in $\partial f_U(\x^0)$. This can be done using $\tilde{\g}^0 \in \partial f(\z^0)$ and orthonormal vectors $\{\su^i\}_{i \in S_1}$. Since $\supp\{\tilde{\g}^0\} \subseteq S_1$, we have $\g^0 = \sum_{i \in S_1} \tilde{g}_i^0 \su^i \in \partial f_U(\x^0)$. Since $\algo \in \ACal_{\det}$ is deterministic, we have $\x^1$ is a function of $\g^0$, and thus $\textsc{Z}_\algo$ can simulate and compute it. To satisfies the support condition $\supp\{\z^1\} \subseteq S_1$, we require that $(\su^i)^\top \x^1 = 0$ for all $i \notin S_1$. Note that we only require the columns of $U$ index by the support $S_1$ to compute $\z^1 = U^\top \x^1$. This confirms the previous argument that $U$ may not be completely determined throughout the implementation of $\textsc{Z}_\algo$ but we can still simulate the operation of $\algo$ on $f_U$. Repeating this process for all $t = 0, 1, 2, \ldots, d-1$, we have shown that $\textsc{Z}_\algo$ can emulate $\x^t$ and from it can construct $\z^t$ that satisfies Eq.~\eqref{result:resisting-second} for all $t \leq d-1$. 

The final step is show that the process can be repeated for all $t = 0, 1, 2, \ldots, d-1$. In particular, after computing $S_{t+1}$, we can find the orthonormal vectors $\{\su^i\}_{i \in S_{t+1} \setminus S_t}$ such that $(\su^j)^\top \x^s = 0$ for all $s \leq t$ and $j \in S_{t+1} \setminus S_t$, and additionally that $U$ is an orthogonal matrix. Equivalently, the orthogonal complement of the span of $\{a^0, a^1, \ldots, a^t, \{u^i\}_{i \in S_t}\}$ is large enough such that we can choose $\{\su^i\}_{i \in S_{t+1} \setminus S_t}$ from it. Note that this orthogonal complement has the dimension at least $2d - (t + 1) - |S_t| = |S_t^c| + d - t - 1 \geq |S_t^c|$ for all $t \leq d-1$ and $|S_t^c| \leq |S_{t+1} \setminus S_t|$. Thus, there exists orthonormal vectors $\{\su^i\}_{i \in S_{t+1} \setminus S_t}$ that meet our requirement. This completes the induction. 
\end{proof}
\paragraph{Proof of Theorem~\ref{Theorem:DET}.} For any algorithm $\algo \in \ACal_{\det}$ and every function $f: \br^d \mapsto \br$, Proposition~\ref{prop:resisting} implies that there exists a zero-respecting algorithm $\textsc{Z}_\algo \in \ACal_{\zr}$ and an orthogonal matrix $U \in \br^{2d \times d}$ (dependent on $f$ and $\algo$) such that 
\begin{equation*}
\bound_{\delta, \epsilon}(\algo, f_U) \geq \min\left\{d + 1, \bound_{\delta, \epsilon}(Z_\algo, f)\right\}. 
\end{equation*}
Letting $f: \br^d \mapsto \br$ be defined in Eq.~\eqref{example:nonconvex}, Theorem~\ref{Theorem:ZR} show that $\TCal_{\delta, \epsilon}(\ACal_{\zr}, \{f\}) \geq T$ for any $T \geq 1$. Since $T$ can be arbitrarily large, we have 
\begin{equation*}
\bound_{\delta, \epsilon}(\textsc{Z}_\algo, f) \geq \inf_{\textsc{Z} \in \ACal_{\zr}} \bound_{\delta, \epsilon}(\textsc{Z}, f) = \TCal_{\delta, \epsilon}(\ACal_{\zr}, \{f\}) \geq T \geq d+1.  
\end{equation*}
Putting these pieces together yields that 
\begin{equation*}
\bound_{\delta, \epsilon}(\algo, f_U) \geq d + 1. 
\end{equation*}
Since $f_U(\x) = f(U^\top \x)$ and $U$ is an orthogonal matrix in $\br^{2d \times d}$, we have 
\begin{equation*}
f_U(\zero_{2d}) - \min_{\x \in \br^{2d}} f_U(\x) = f(\zero_d) - \min_{\x \in \br^d} f(\x), \quad \textnormal{and} \quad \|f_U(\x') - f_U(\x)\| = \|f(U^\top\x') - f(U^\top\x)\|. 
\end{equation*}
Since $f \in \FCal_d(\Delta, L)$, we have $f_U(\zero_{2d}) - \min_{\x \in \br^{2d}} f_U(\x) \leq \Delta$ and $\|f_U(\x') - f_U(\x)\| \leq L\|U^\top(\x' - \x)\| \leq L\|\x' - \x\|$. This implies that $f_U \in \FCal_{2d}(\Delta, L)$. Putting these pieces together yields that 
\begin{equation*}
\sup_{f \in \FCal_{2d}(\Delta, L)} \bound_{\delta, \epsilon}(\algo, f) \geq d + 1. 
\end{equation*}
By taking the infimum over $\algo \in \ACal_{\det}$, we conclude that 
\begin{equation*}
\TCal_{\delta, \epsilon}(\ACal_{\det}, \FCal_{2d}(\Delta, L)) = \inf_{\algo \in \ACal_{\det}} \sup_{f \in \FCal_{2d}(\Delta, L)} \bound_{\delta, \epsilon}(\algo, f) \geq d + 1. 
\end{equation*}
This completes the proof. 

\section{Deterministic Algorithm with Only $1^\textnormal{st}$ Oracle}\label{sec:grad}
In this section, we demonstrate the importance of having access either to randomness or to an $0^\textnormal{th}$ oracle. In particular, we prove that any deterministic algorithm with only $1^\textnormal{st}$ oracle can not find an approximate Goldstein stationary point within a finite number of iterations up to small constant tolerances. 
\begin{theorem} \label{Theorem:1LB}
For any deterministic algorithm with only $1^\textnormal{st}$ oracle, there exists an $1$-Lipschitz function $f : \br^d \to [-1, 1]$ such that the algorithm can not find a $(\delta, \epsilon)$-Goldstein stationary point in a finite number of iterations for any $0 < \delta < \epsilon < 1$.
\end{theorem}
\begin{proof}
Suppose that we are accessible to only $1^\textnormal{st}$ oracle of a function $f : \br \to [-1, 1]$, the deterministic algorithm $\ACal$ can query a sequence of points $Q = \{\x_1, \x_2, \ldots, \x_m\}$ for any fixed and finite integer $m \geq 1$ starting from querying $\x_1$. For all of given query points $\x_i$, we always have $\nabla f(\x_i) = 1$. After receiving these uninformative answers (i.e., $\nabla f(\x) = 1$ for all $\x \in Q$), the algorithm returns the candidate solution $\hat{\x}$ for being a $(\delta, \epsilon)$-Goldstein stationary point. In addition, we remark that $\hat{\x}$ might not be in $Q$. 

It suffices to construct a 1-Lipschitz function $f$ such that $\nabla f(\x_i) = 1$ for all $1 \leq i \leq m$ and $\hat{\x}$ is not a $(\delta, \epsilon)$-Goldstein stationary point given that $0 < \delta < \epsilon < 1$. Our strategy to achieve the second goal is simple and intuitive. Indeed, we set $f(\x) = \x - \hat{\x}$ for all $\x \in [\hat{\x} - \delta + \eta, \hat{\x} + \delta + \eta]$ for some small positive value $\eta < 1 - \delta$ chosen so that $\hat{\x} + \delta + \eta \notin Q$ and $\hat{\x} - \delta - \eta \notin Q$. Using the definition of $\delta$-Goldstein subdifferential of $f$, we have $\partial_\delta f(\hat{\x}) = \{1\}$ and the norm of the minimal-norm element in $\partial_\delta f(\hat{\x})$ is 1. Since $\epsilon < 1$, we have $\hat{\x}$ is not a $(\delta, \epsilon)$-Goldstein stationary point for any $0 < \delta < \epsilon < 1$. The second goal is satisfied. Moreover, for all $\x \in Q \cap [\hat{\x} - \delta + \eta, \hat{\x} + \delta + \eta]$, we have $\nabla f(\x) = 1$. Thus, for these query points that lie in the interval $[\hat{\x} - \delta + \eta, \hat{\x} + \delta + \eta]$, the first goal is satisfied. 
  
It remains to define the function $f(\x)$ for any $\x \in (\hat{\x} + \delta + \eta, +\infty)$. The idea is simply keeping $f(\x) = \delta + \eta$ in this range while adding some small bumps to guarantee that $\nabla f'(\x) = 1$ is satisfied for all $\x \in Q \cap (\hat{\x} + \delta + \eta, +\infty)$. Let $\bar{Q} = Q \cup \{\hat{\x} - \delta + \eta, \hat{\x} + \delta + \eta\}$ and $r_1 = \frac{1}{10} \min_{\x,\x' \in \bar{Q}, \x \neq \x'}\{|\x - \x'|\}$, we define $r = \min\{r_1, \delta\}$ and 
\begin{equation*}
f(\x) = \left\{ 
\begin{aligned} 
\delta + \eta ~~~~~ & |\x - \x'| > r \textnormal{ for any } \x' \in Q, \\
\delta + \eta - \x ~~~~~ & \exists \x' \in Q \text{ s.t. } |\x - \x'| \leq r \textnormal{ and } \x \leq \x' - \tfrac{r}{2}, \\
\delta + \eta - r + \x ~~~~~ & \exists \x' \in Q \text{ s.t. } |\x - \x'| \leq r \textnormal{ and } \x > \x' - \tfrac{r}{2}. 
\end{aligned}
\right. 
\end{equation*}
We see from the above definition that $0 \leq f(\x) \leq \delta + \eta$ for all $\x \in (\hat{\x} + \delta + \eta, \infty)$ and $\nabla f(\x) = 1$ for all $\x \in Q \cap (\hat{\x} + \delta + \eta, \infty)$. Using the similar approach, we define $f(\x)$ for any $\x \in (-\infty,  \hat{\x} - \delta - \eta)$ as:
\begin{equation*}
f(\x) = \left\{ 
\begin{aligned}
- \delta - \eta ~~~~~ & |\x - \x'| > r \textnormal{ for any } \x' \in Q, \\
- \delta - \eta + \x ~~~~~ & \exists \x' \in Q \text{ s.t. } |\x - \x'| \le r \textnormal{ and } \x \leq \x' + \tfrac{r}{2}, \\
- \delta - \eta + r - \x ~~~~~ & \exists \x' \in Q \text{ s.t. } |\x - \x'| \le r \textnormal{ and } \x > \x' + \tfrac{r}{2}. 
\end{aligned}
\right. 
\end{equation*}
Putting these pieces together, we conclude that the 1-Lipschitz function $f$ satisfies that $\nabla f(\x_i) = 1$ for all $1 \leq i \leq m$ and $\hat{\x}$ is not a $(\delta, \epsilon)$-Goldstein stationary point for any $0 < \delta < \epsilon < 1$. This together with the fact that $m \geq 1$ is arbitrarily chosen yields the desired result. 
\end{proof}
\begin{remark} \label{rem:randomized1stOracle}
The finite-time convergence guarantee is achieved for the randomized algorithms with only $1^\textnormal{st}$ oracle. Indeed,~\citet[Theorem~3.1]{Lin-2022-Gradient} has shown that $\nabla f_\delta(\x) = \EE_{\su \in \mathbb{P}}[\nabla f(\su)] \in \partial_\delta f(\x)$ where $f_\delta(\x) = \EE_{\su \in \mathbb{P}}[f(\su)]$ and $\mathbb{P}$ is an uniform distribution on a unit $\ell_2$-ball centered in $\x$. Thus, it suffices to find an $\epsilon$-stationary point of $f_\delta$. Since $f$ is $L$-Lipschitz, $\nabla f(\su)$ is an unbiased estimator of $\nabla f_\delta(\x)$ and $\|\nabla f(\su)\| \leq L$ for any $\su \sim \mathbb{P}$. By using the similar arguments for proving~\citet[Theorem~3.2]{Lin-2022-Gradient}, we conclude that the required number of $1^\textnormal{st}$ oracle is bounded by $O(\sqrt{d}(L^4\epsilon^{-4} + \Delta L^3\delta^{-1}\epsilon^{-4}))$. 
\end{remark}

\section{Deterministic Smoothing and Complexity Analysis}\label{sec:smoothing}
In this section, we present a deterministic smoothing approach that achieves the smoothness parameter that is exponential in some parameters $M > 0$, and propose a deterministic algorithm that achieves the complexity bound of $\tilde{O}(M\delta^{-1}\epsilon^{-3})$. The scheme is inspired by discrete gradient method~\citep{Bagirov-2003-Continuous, Bagirov-2008-Discrete} and the asymptotic convergence of the generated iterates to a Goldstein stationary point of a locally Lipschitz function has been proven by~\citet{Mahdavi-2012-Effective}.  

\subsection{Deterministic smoothing of arithmetic circuits}
We introduce a smoothing technique that can be applied if we have the access to the \textit{arithmetic circuit} that captures the entire structure of the function. This model of representing functions as arithmetic circuits has been commonly used across diverse domains, from purely theoretical applications (e.g., computational complexity theory~\citep{Daskalakis-2011-Continuous, Fearnley-2021-Complexity}) to practical applications (e.g., deep neural networks~\citep{Lecun-2015-Deep, Goodfellow-2016-Deep}). In this context, we refer to it as a \textit{white box} model. Our results demonstrate that having such the access to the circuits is powerful enough to allow for deterministic smoothing.

\begin{definition}[Linear Arithmetic Circuits~\citep{Fearnley-2021-Complexity}]\label{def:AC}
We say that $\calC$ is a linear arithmetic circuit if it is represented as a directed acyclic graph with three different group of nodes: (i) input nodes; (ii) output nodes; and (iii) gate nodes. The gate node can be one of $\{+, \max, \times \zeta, \mathrm{const(c)}\}$, where $\mathrm{const(c)}$ stands for being a constant $c \in [-1, 1]$ and $\times \zeta$ stands for the multiplication by a constant $\zeta \in [-1, 1]$\footnote{We can generalize this to any bounded valued constants but we keep the range of constants to $[-1, 1]$ for simplicity.}. Moreover, a valid arithmetic circuit also satisfies the following conditions: 
\begin{enumerate}
\item There can be more than one input node, where each one has 0 incoming edges but any number of outgoing edges.  
\item The gate nodes in $\{+, \max\}$ have two incoming edges and any number of outgoing edges\footnote{We can generalize it to the case of any finite number of inputs. Focusing on two incoming edges does not lack the generality since we can always compose these gates to simulate addition and maximum with many inputs by just increasing the size and the depth of the circuit by a logarithmic factor.}. 
\item The gate node $\mathrm{const(c)}$ has 0 incoming edges but any number of outgoing edges. 
\item The gate node $\times \zeta$ has 1 incoming edge but any number of outgoing edges. 
\item There is only one output node that has 1 incoming edge and 0 outgoing edge.
\end{enumerate}
Throughout this section, we refer to $s(\calC)$ as the \textit{size} of $\calC$ (i.e., the number of nodes in the graph of $\calC$) and refer to $p(\calC)$ as the \textit{depth} of $\calC$ (i.e., the length of the longest path of the graph of $\calC$).
\end{definition}
The interpretation of $\calC$ as a function $f : \br^d \to \br$ is intuitive. The input nodes correspond to the input variables $x_1, \ldots, x_d$, then every gate node defines an arithmetic operation over these variables and we finally output the result $f(\x)$ in the output node. 

Despite the limited types of gates, the class of functions that we can characterize in this way is huge. In particular, we can use this representation to approximate up to error $\epsilon$ any efficiently computable function over a bounded but maybe exponential domain with the size that scales only as $\mathrm{poly}(\log(1/\epsilon))$ so that $\epsilon > 0$ can even be exponentially small; see the proof in~\citet[Appendix E]{Fearnley-2021-Complexity}. This means that we can represent the exponential function or the logarithmic function or any neural network using these type of circuits. The problem of finding a $(\delta, \epsilon)$-Goldstein stationary point of a function $f$ that is represented by an arithmetic circuit $\calC$ falls in the nonsmooth nonconvex optimization framework since the function that we can describe are nonsmooth due to the use of the $\max$ gate. To that end, the above framework captures a wide range of nonsmooth and nonconvex problems.

Our deterministic smoothing idea is to replace the $\max$ function with its smooth alternative $\softmax$ that we define below. This guarantees that the resulting function is smooth. If we carefully choose the smoothness parameter of $\softmax$, the original function and its deterministically smoothed version will output the same value up to an exponentially small error $\gamma > 0$. The main issue with this reduction is that the smoothness of the resulting function is exponentially large in some parameter $M$ and hence the well-known algorithms for smooth optimization (e.g., gradient descent) suffers from such exponentially large dependence on $M > 0$. This motivates the algorithm that we present in the next section that achieves the logarithmic dependence on the smoothness parameter. Combined this algorithm with the deterministic smoothing yields a new deterministic algorithm where the complexity bound for finding a $(\delta, \epsilon)$-Goldstein stationary point of the smoothed version will be $\tilde{O}(M\delta^{-1}\epsilon^{-3})$. 

It is worth remarking that our smoothing procedure is simple, implementable and inspired by the techniques that have been widely accepted in practice. To prove its efficiency, we need to impose the following assumption that is satisfied by the practical design of deep neural networks (see Remark \ref{rem:designNNs}).
\begin{assumption} \label{asp:recursiveLipschitz}
Let $f : \br^d \to \br$ that is represented as a linear arithmetic circuit $\calC$. Let also $v_1, \dots, v_n$ be the nodes that correspond to nodes in $\calC$ and $f_i$ be the function that will be computed if $v_i$ would the output of the arithmetic circuit. Then, we assume that the Lipschitzness $L_i$ of the functions $f_1, \dots, f_n$ is given recursively according to the following rules:
\begin{enumerate}
\item[-] \textbf{$\boldsymbol{v_i}$ is a $\boldsymbol{+}$ gate:} if $f_i = f_j + f_k$ then $L_i = L_j + L_k$. 
\item[-] \textbf{$\boldsymbol{v_i}$ is a $\boldsymbol{\max}$ gate:} if $f_i = \max\{f_j, f_k\}$ then $L_i = \max\{L_j, L_k\}$. 
\item[-] \textbf{$\boldsymbol{v_i}$ is a $\boldsymbol{\mathrm{const}(c)}$ gate:} $L_i = 0$. 
\item[-] \textbf{$\boldsymbol{v_i}$ is a $\boldsymbol{\times \zeta}$ gate:} if $f_i = \zeta \cdot f_j$ then $L_i = |\zeta| \cdot L_j$. 
\item[-] \textbf{$\boldsymbol{v_i}$ is a input node:} $L_i = 1$.
\end{enumerate}
In particular, we assume that the Lipschitzness of $f$, which is equal to the evaluation of the last node, is given by the above recursion. In this case, we say that $f$ is \textit{$L$-recursively Lipschitz}.
\end{assumption}
\begin{remark}\label{rem:designNNs}
We remark that that the recursive rules used in Assumption \ref{asp:recursiveLipschitz} always provide an upper bound on $L > 0$ which could be however much larger than $L$ in the worst-case. To bypass these bad cases, we impose Assumption \ref{asp:recursiveLipschitz} which is crucial to the proof of Theorem \ref{thm:circuitSmoothing}. Notably, this assumption is not theoretically artificial but can be satisfied by the generic construction of neural networks in the context of deep learning. Indeed, the $\boldsymbol{+}$ and $\boldsymbol{\times \zeta}$ gates are often consecutively used in the construction of neural networks and leads to $L_i = |\zeta_j| \cdot L_j + |\zeta_k| \cdot L_k$. To stabilize the training, practitioners often force $L_i = \min\{|\zeta_j| \cdot L_j + |\zeta_k| \cdot L_k, 1\}$ by employing the normalization techniques~\citep{Ioffe-2015-Batch, Miyato-2018-Spectral}. In addition, both $\boldsymbol{\max}$ and $\textnormal{\bf softmax}$ gates guarantee that $L_i = 1$ if $L_j = L_k = 1$. Thus, we have $L = 1$ and the induced function $f: \br^d \mapsto \br$ is $1$-recursively Lipschitz. 
\end{remark}

\begin{theorem}\label{thm:circuitSmoothing}
Let $f : \br^d \to \mathbb{R}$ be a $L$-recursively Lipschitz function (see Assumption \ref{asp:recursiveLipschitz}), represented by a linear arithmetic circuit $\calC$. For every $N \in \br$, we can construct a function $g : \br^d \to \br$ such that:
\begin{enumerate}
\item $|f(\x) - g(\x)| \le 2^{-N}$. 
\item $g$ is $L$-Lipschitz. 
\item $g$ is $2^{3 s(\calC) + N}$-smooth.
\end{enumerate}
\end{theorem}
\begin{proof}
We use exactly the same arithmetic circuit for the evaluation of $g$ except for replacing all the $\boldsymbol{\max}$ gates with the $\textnormal{\bf softmax}$ gates: $\softmax_{a}(z_1, z_2) = \frac{1}{a} \ln(\exp(a \cdot z_1) + \exp(a \cdot z_2))$. Before the formal argument, we summarize the properties of softmax gates:
\begin{lemma} \label{lem:softmax}
We have that (i) $|\max(z_1, z_2) - \softmax_{a}(z_1, z_2)| \leq \frac{1}{a}$, (ii) $\softmax_a(\cdot)$ is $1$-Lipschitz, and (iii) $\softmax_a$ is $\frac{1}{2}a$-smooth.
\end{lemma}
\begin{proof}
For the first part, we assume without loss of generality that $z_1 \geq z_2$ and obtain that $h(z_2) = z_1 - \softmax_{a}(z_1, z_2) = |\max(z_1, z_2) - \softmax_{a}(z_1, z_2)|$. Here the last equality holds since $z_1 \geq z_2$. Since the derivative of $h$ with respect to $z_2$ is positive, we have the maximum value for $h$ will be achieved only when $z_2 = z_1$. In addition, $\softmax_a(z_1, z_1) = z_1 + \frac{\ln(2)}{a}$. Putting these pieces together yields the first part of the lemma. The second and third parts are just a matter of algebraic calculations.
\end{proof}
We are ready to define a topological sorting of $\calC$ and start to compare the evaluation of the nodes in the circuit of $f$ and in the circuit of $g$ under Assumption~\ref{asp:recursiveLipschitz}. Let $f_i$ be the function evaluated in the node $i$ of the circuit of $f$ and $g_i$ the function evaluated in the node $i$ of the circuit of $g$. Let also $L_i > 0$ be the corresponding Lipschitzness parameter of $f_i$. 

The input nodes and the constant gates for $f$ and $g$ both have the same value and the simple gradient being $e_k$ for some $k$. Thus, they are $0$-smooth. Also, the input nodes are $1$-Lipschitz and the constant nodes are $0$-Lipschitz. These input gates are the basis of our induction. Our inductive hypothesis is: for all $j < i$, we assume that (i) $|f_j(\x) - g_j(\x)| \leq \gamma_j$ for all $\x \in \br^d$, (ii) $g_j$ is $S_j$-smooth, and (iii) $g_j$ is $L_j$-Lipschitz. Then, we try to prove that (i) $|f_i(\x) - g_i(\x)| \leq \gamma_i$ for all $\x \in \br^d$, (ii) $g_i$ is $S_i$-smooth, and (iii) $g_i$ is $L_i$-Lipschitz. In the following, we consider different cases for the type of node $i$: 
\begin{itemize}
\item \textbf{output node.} In this case, the value and Lipschitzness of node $i$ is the same as of a node $j < i$. Thus, we have $|f_i(\x) - g_i(\x)| \le \gamma_j \triangleq \gamma_i$ and obtain that $S_i = S_j$ and $L_i = L_j$.
\item \textbf{$\times \zeta$ node.} There exists $j < i$ such that $f_i(\x) = \zeta f_j(\x)$ and $g_i(\x) = \zeta g_j(\x)$ which means that $|f_i(\x) - g_i(\x)| \le |\zeta| \cdot \gamma_j \le \gamma_j \triangleq \gamma_i$. Also, $S_i \le |\zeta| \cdot S_j \le S_j$. Then, the Lipschitzness of $g_i$ is upper bounded by $|\zeta| \cdot L_j$ which is equal to $L_i$ by Assumption \ref{asp:recursiveLipschitz}. Thus, $g_i$ is $L_i$-Lipschitz.
\item \textbf{$+$ node.} There exist $j, k < i$ such that $f_i(\x) = f_j(\x) + f_k(\x)$ and $g_i(\x) = g_j(\x) + g_k(\x)$ which means that $|f_i(\x) - g_i(\x)| \le \gamma_j + \gamma_k \triangleq \gamma_i$. Also, $S_i \le S_j + S_k$. Then, the Lipschitzness of $g_i$ is upper bounded by $L_j + L_k$ which is equal to $L_i$ by Assumption \ref{asp:recursiveLipschitz}. Thus, $g_i$ is $L_i$-Lipschitz.
\item \textbf{$\max$ node.} There exist $j, k < i$ such that $f_i(\x) = \max(f_j(\x), f_k(\x))$ and $g_i(\x) = \softmax(g_j(\x), g_k(\x))$. Using Lemma \ref{lem:softmax}, the triangle inequality and the fact that $\max$ is $1$-Lipschitz, we have $|f_i(\x) - g_i(\x)| \le \frac{1}{a} + \gamma_j + \gamma_k \triangleq \gamma_i$. The next is to bound the smoothness $S_i$. By definition, we have
\begin{equation*}
\nabla g_i(\x) = [\nabla \softmax_a (g_j(\x), g_k(\x))]^\top[\nabla g_j(\x) ~~ \nabla g_k(\x)]. 
\end{equation*}
This implies that 
\begin{eqnarray*}
\lefteqn{\norm{\nabla g_i(\x) - \nabla g_i(\y)}} \\ 
& = & \|[\nabla \softmax_a (g_j(\x), g_k(\x))]^\top [\nabla g_j(\x) ~~ \nabla g_k(\x)] - [\nabla \softmax_a (g_j(\y), g_k(\y))]^\top[\nabla g_j(\y) ~~ \nabla g_k(\y)]\| \\
& = & \|[\nabla \softmax_a (g_j(\x), g_k(\x))]^\top [\nabla g_j(\x) ~~ \nabla g_k(\x)] - [\nabla \softmax_a (g_j(\x), g_k(\x))]^\top [\nabla g_j(\y) ~~ \nabla g_k(\y)]\| \\
& & + \|[\nabla \softmax_a (g_j(\x), g_k(\x))]^\top [\nabla g_j(\y) ~~ \nabla g_k(\y)] - [\nabla \softmax_a (g_j(\y), g_k(\y))]^\top[\nabla g_j(\y) ~~ \nabla g_k(\y)]\| \\
& \leq & (S_j + S_k + \tfrac{1}{2}a(L_j + L_k))\|\x - \y\|.  
\end{eqnarray*}
The last inequality holds true since $\softmax$ is $1$-Lipschitz and $\frac{a}{2}$-smooth and the Lipschitz constants of $\nabla g_j$ and $\nabla g_k$ are $L_j$ and $L_k$. A trivial upper bound from~\citet{Fearnley-2021-Complexity} implies that $L_j \leq 2^j$ for all $j$. Thus, we have $S_i \le a \cdot 2^{i - 1} + S_j + S_k$. Finally, due to the fact that $\softmax$ is $1$-Lipschitz, the Lipschitzness of $g_i$ is upper bounded by $\max\{L_j, L_k\}$ which is equal to $L_i$ by Assumption \ref{asp:recursiveLipschitz}. Thus, $g_i$ is $L_i$-Lipschitz.
\end{itemize}
Note that the sequence of errors $\{\gamma_i\}_{i \geq 1}$ is increasing. Given all above cases that we consider,  we have $\gamma_i \leq \frac{1}{a} + \gamma_j + \gamma_k \leq \frac{1}{a} + 2 \gamma_{i - 1}$ which implies that $\gamma_i \le \frac{2^i}{a}$. Thus, we have $|f(\x) - g(\x)| \leq \frac{2^{s(\calC)}}{a}$. The similar argument guarantees that $S_i \leq a \cdot 2^{s(\calC)} + 2 S_{i - 1}$ and hence $S_i \le a \cdot 2^{s(\calC) + i}$. Thus, we have $S_i \leq a \cdot 2^{2 s(\calC)}$. If we choose $a = 2^{- N - s(\calC)}$, Part 1 and Part 3 of Theorem~\ref{thm:circuitSmoothing} follow. In addition, since $g_i$ is $L_i$-Lipschitz for all $i \geq 1$ where $L_i$ is the Lipschitz constant of $f_i$ and $f$ is $L$-Lipschitz, we have $g$ is $L$-Lipschitz. 
\end{proof}
\begin{remark}
Theorem~\ref{thm:circuitSmoothing} illustrates that the smoothing of a Lipschitz function can be done in a simple and deterministic manner if we are accessible to its representation using linear arithmetic circuits. However, the resulting smooth function has an exponentially large smoothness parameter. Fortunately, we show in the next subsection that it is not a matter since there exists a deterministic algorithm that can achieve the logarithmic smoothness dependence. 
\end{remark}

\subsection{Deterministic algorithm with complexity bound guarantee}
Focusing on the smooth and nonconvex optimization problems where $f \in \FCal_d(\Delta, L)$ is $\Theta(2^M)$-smooth for some parameter $M > 0$, we develop a simple deterministic algorithm that achieves the complexity bound of $\tilde{O}(M\delta^{-1}\epsilon^{-3})$ in terms of $1^\textnormal{st}$ and $0^\textnormal{th}$ oracles. 

We first give a brief overview of Goldstein's method. Indeed, let $\partial_\delta f(\x) := \textnormal{conv}(\cup_{\y \in \BB_\delta(\x)}  \partial f(\y))$ be the $\delta$-Goldstein subdifferential of a Lipschitz function $f$ at $\x$, we define the minimal-norm element: 
\begin{equation*}
\g(\x) = \argmin\left\{\|\g\|: \g \in \partial_\delta f(\x)\right\}. 
\end{equation*}
At each iteration of Goldstein's method, we update $\x^+ \leftarrow \x - \delta(\g(\x)/\|\g(\x)\|)$. Since $f$ is differentiable, the mean-value theorem implies that 
\begin{equation*}
f(\x^+) - f(\x) = f(\x - \delta\tfrac{\g(\x)}{\|\g(\x)\|}) - f(\x) = -\delta\xi^\top\left(\tfrac{\g(\x)}{\|\g(\x)\|}\right), \quad \textnormal{for some } \xi \in \partial_\delta f(\x).  
\end{equation*}
Since $\g(\x)$ is the minimal-norm element in $\partial_\delta f(\x)$, we have $(\xi - \g(\x))^\top\g(\x) \geq 0$. This implies that 
\begin{equation*}
f(\x^+) - f(\x) \leq -\delta\left(\tfrac{\|\g(\x)\|^2}{\|\g(\x)\|}\right) = -\delta\|\g(\x)\|. 
\end{equation*}
As such, the number of iterations to find a $(\delta, \epsilon)$-Goldstein stationary point is bounded by $O(\Delta\delta^{-1}\epsilon^{-1})$. The drawback of Goldstein's method is that we can not compute the minimal-norm element in $\partial_\delta f(\x)$ exactly in general. While all of randomized first-order algorithms~\citep{Zhang-2020-Complexity, Davis-2022-Gradient, Tian-2022-Finite} employ different strategies to approximate $\partial_\delta f(\x)$, we consider approximating $\partial_\delta f(\x)$ using a convex hull of the finite number of elements in $\partial_\delta f(\x)$ \textit{deterministically}. More specifically, let $W = \{\g_1, \g_2, \ldots, \g_k\} \subseteq \partial_\delta f(\x)$, we consider using $\conv(W)$ as an approximation of $\partial_\delta f(\x)$ and compute 
\begin{equation*}
\g(\x) = \argmin\left\{\|\g\|: \g \in \conv(W)\right\}. 
\end{equation*}
If $k$ is sufficiently large and $W$ approximates $\partial_\delta f(\x)$ well, we have the following condition holds true: 
\begin{equation}\label{condition:descent}
f(\x - \delta\tfrac{\g(\x)}{\|\g(\x)\|}) - f(\x) \leq - \tfrac{\delta}{2}\|\g(\x)\|. 
\end{equation}
Otherwise, $W$ is not sufficiently large and we need to improve the approximation of $\partial_\delta f(\x)$ by updating $W \leftarrow W \cup \{\g_{\textnormal{new}}\}$ where $\g_{\textnormal{new}} \in \partial_\delta f(\x)$ and $\g_{\textnormal{new}} \notin \conv(W)$. 
\begin{algorithm}[!t]
\begin{algorithmic}\caption{\textsf{Binary-Search}($\delta$, $\nabla f(\cdot)$, $\g_0$, $\x$)}\label{Alg:BS}
\STATE \textbf{Initialization:} Set $b \leftarrow \delta$, $a \leftarrow 0$ and $t \leftarrow b$.
\REPEAT
\STATE Compute $\nabla h(t) = -(\nabla f(\x - t\tfrac{\g_0}{\|\g_0\|}))^\top\tfrac{\g_0}{\|\g_0\|} + \tfrac{1}{2}\|\g_0\|$. 
\IF{$\nabla h(t) < -\frac{\epsilon}{4}$}
\STATE Set $t \leftarrow \frac{a+b}{2}$.  
\ENDIF
\IF{$h(b) > h(t)$}
\STATE Set $a \leftarrow t$.  
\ELSE
\STATE Set $b \leftarrow t$. 
\ENDIF
\UNTIL{$\nabla h(t) \geq -\frac{\epsilon}{4}$}
\STATE \textbf{Output:} $\nabla f(\x - t\frac{\g_0}{\|\g_0\|})$. 
\end{algorithmic}
\end{algorithm}
The next step is to study how to select $\g_{\textnormal{new}}$ and bound the number of $1^\textnormal{st}$ and $0^\textnormal{th}$ oracles required to select it. In particular, suppose that Eq.~\eqref{condition:descent} does not hold true, i.e., $f(\x - \delta\tfrac{\g(\x)}{\|\g(\x)\|}) - f(\x) > - \frac{\delta}{2}\|\g(\x)\|$, we define the one-dimensional function $h: \br \mapsto \br$ as follows, 
\begin{equation*}
h(t) = f(\x - t\tfrac{\g(\x)}{\|\g(\x)\|}) - f(\x) + \tfrac{t}{2}\|\g(\x)\|. 
\end{equation*} 
Since $\g(\x)$ is the minimum-norm element in $\conv(W)$, we have $(\xi - \g(\x))^\top\g(\x) \geq 0$ for all $\xi \in \conv(W)$. Equivalently, $\xi^\top\g(\x) \geq \|\g(\x)\|^2$. Thus, if we compute $\g_{\textnormal{new}} \in \partial_\delta f(\x)$ such that $\g_{\textnormal{new}}^\top\g(\x) \leq \frac{3}{4}\|\g(\x)\|^2$, we have $\g_{\textnormal{new}} \notin \conv(W)$. 

Since $f$ is $\Theta(2^M)$-smooth for some parameter $M > 0$, we have $h$ is $\Theta(2^M)$-smooth. Suppose that there exists $t_0 \in (0, \delta)$ satisfying $\nabla h(t_0) \geq -\frac{\epsilon}{4}$, we obtain from $\|\g(\x)\| \geq \epsilon$ that 
\begin{equation*}
-(\nabla f(\x - t_0\tfrac{\g(\x)}{\|\g(\x)\|}))^\top\tfrac{\g(\x)}{\|\g(\x)\|} + \tfrac{1}{2}\|\g(\x)\| \geq -\tfrac{\epsilon}{4} \geq -\tfrac{\|\g(\x)\|}{4}. 
\end{equation*}
Equivalently, we have 
\begin{equation*}
(\nabla f(\x - t_0\tfrac{\g(\x)}{\|\g(\x)\|}))^\top\g(\x) \leq \tfrac{3}{4}\|\g(\x)\|^2. 
\end{equation*}
It suffices to find $t_0 \in (0, \delta)$ satisfying $\nabla h(t_0) \geq -\frac{\epsilon}{4}$ and set $\g_{\textnormal{new}} = \nabla f(\x - t_0\tfrac{\g(\x)}{\|\g(\x)\|})$. By the definition, we have $h(0) = 0$. Since $f(\x - \delta\tfrac{\g(\x)}{\|\g(\x)\|}) - f(\x) > - \frac{\delta}{2}\|\g(\x)\|$, we have $h(\delta) > 0$. Putting these pieces together yields the existence of $t_0 \in (0, \delta)$ satisfying $\nabla h(t_0) \geq -\frac{\epsilon}{4}$. In addition, $f$ is $\Theta(2^M)$-smooth for some parameter $M > 0$. Then, the binary search scheme can find such $t_0 \in (0, \delta)$ within $\Theta(M\log(\delta/\epsilon))$ number of $1^\textnormal{st}$ and $0^\textnormal{th}$ oracles. For the sake of completeness, we summarize the scheme in Algorithm~\ref{Alg:BS}. 
\begin{algorithm}[!t]
\begin{algorithmic}[1]\caption{\textsf{Modified-Goldstein-SG}($\x_0$, $\delta$, $\epsilon$, $T$)}\label{Alg:Full}
\STATE \textbf{Input:} initial point $\x_0 \in \br^d$, tolerances $\delta, \epsilon \in (0, 1)$ and iteration number $T$. 
\STATE \textbf{Initialization:} Set $\x_0 \in \br^d$. 
\FOR{$k = 0, 1, 2, \ldots, T-1$}
\STATE Compute $\g_\textnormal{initial} \in \partial f(\x_k)$.  
\STATE Set $W \leftarrow \{\g_\textnormal{initial}\}$. 
\REPEAT
\STATE Set $\g(\x_k) \leftarrow \argmin\left\{\|\g\|: \g \in \conv(W)\right\}$. 
\STATE Set $\g_{\textnormal{new}} \leftarrow \textsf{Binary-Search}(\delta, \nabla f(\cdot), \g(\x_k), \x_k)$. 
\STATE Set $W \leftarrow W \cup \{\g_{\textnormal{new}}\}$. 
\UNTIL{$f(\x_k - \delta\tfrac{\g(\x_k)}{\|\g(\x_k)\|}) - f(\x_k) \leq - \tfrac{\delta}{2}\|\g(\x_k)\|$ or $\|\g(\x_k)\| \leq \epsilon$.}
\IF{$\|\g(\x_k)\| \leq \epsilon$}
\STATE \textbf{Stop}.
\ELSE
\STATE $\x_{k+1} \leftarrow \x_k - \delta\frac{\g(\x_k)}{\|\g(\x_k)\|}$. 
\ENDIF
\ENDFOR
\STATE \textbf{Output:} $\x_k$. 
\end{algorithmic}
\end{algorithm}

It remains to bound the number of selecting $\g_{\textnormal{new}}$ required to construct $W$ such that it approximates $\partial_\delta f(\x)$ well and Eq.~\eqref{condition:descent} holds true. In particular, suppose that Eq.~\eqref{condition:descent} does not hold true, we can compute $\g_{\textnormal{new}} \in \partial_\delta f(\x)$ such that $\g_{\textnormal{new}}^\top\g(\x) \leq \frac{3}{4}\|\g(\x)\|^2$. For simplicity, we define 
\begin{equation*}
\g_\textnormal{new}(\x) = \argmin\left\{\|\g\|: \g \in \conv(W \cup \{\g_{\textnormal{new}}\})\right\}. 
\end{equation*}
For all $t \in (0, 1)$, we have 
\begin{equation*}
\|\g_\textnormal{new}(\x)\|^2 \leq \|\g(\x) + t(\g_{\textnormal{new}} - \g(\x))\|^2 = \|\g(\x)\|^2 + 2t\g(\x)^\top(\g_{\textnormal{new}} - \g(\x)) + t^2\|\g_{\textnormal{new}} - \g(\x)\|^2. 
\end{equation*}
Since $f$ is $L$-Lipschitz, we have $\|\g_{\textnormal{new}} - \g(\x)\|^2 \leq 4L^2$. Putting these pieces together yields that 
\begin{equation*}
\|\g_\textnormal{new}(\x)\|^2 \leq (1 - \tfrac{t}{2})\|\g(\x)\|^2 + 4t^2L^2. 
\end{equation*}
Since $\|\g(\x)\| \leq L$, we set $t = \frac{\|\g(\x)\|^2}{16L^2} \in (0, 1)$. This together with the fact that $\|\g(\x)\| \geq \epsilon$ yields that 
\begin{equation*}
\|\g_\textnormal{new}(\x)\|^2 \leq (1 - \tfrac{\epsilon^2}{64L^2})\|\g(\x)\|^2. 
\end{equation*}
The above contraction inequality together with the facts that $\epsilon \leq \|\g_\textnormal{new}(\x)\|, \|\g(\x)\| \leq L$ implies that the number of selecting $\g_{\textnormal{new}}$ required to get an approximation of $\partial_\delta f(\x)$ satisfying Eq.~\eqref{condition:descent} is bounded by $O(L^2\epsilon^{-2}\log(L/\epsilon))$. We present the detailed scheme of our algorithm in Algorithm~\ref{Alg:Full} and summarize our results in the following theorem. 
\begin{theorem}\label{Theorem:SMOOTH}
Suppose that $f \in \FCal_d(\Delta, L)$ is $\Theta(2^M)$-smooth for some parameter $M > 0$ and let $\epsilon, \delta \in (0, 1)$, there exists some $T > 0$ such that the output $\hat{\x} = \textsf{Modified-Goldstein-SG}(\x_0, \delta, \epsilon, T)$ will satisfy that $\hat{\x}$ is a $(\delta, \epsilon)$-Goldstein stationary point and the required number of $1^\textnormal{st}$ and $0^\textnormal{th}$ oracles is bounded by
\begin{equation*}
O\left(\frac{\Delta L^2 M}{\delta\epsilon^3}\log\left(\frac{L}{\epsilon}\right)\log\left(\frac{\delta}{\epsilon}\right)\right), 
\end{equation*}
where the problem parameters $\Delta, L > 0$ are both independent of the dimension $d \geq 1$. 
\end{theorem}
\begin{proof}
For the outer loops of Algorithm~\ref{Alg:Full}, the number of iterations required to find a $(\delta, \epsilon)$-Goldstein stationary point is bounded by $O(\Delta\delta^{-1}\epsilon^{-1})$. For each outer loop, the number of selecting new elements required to get an approximation of $\partial_\delta f(\x_k)$ satisfying Eq.~\eqref{condition:descent} is bounded by $O(L^2\epsilon^{-2}\log(L/\epsilon))$. For selecting each new element in $\textsf{Binary-Search}$, the required number of $1^\textnormal{st}$ and $0^\textnormal{th}$ oracles is bounded by $O(M\log(\delta/\epsilon))$. Putting these pieces together yields the desired results. 
\end{proof}

\section{Concluding Remarks}\label{sec:conclu}
In this paper, we provide the lower and upper bounds on the complexity of finding an approximate Goldstein stationary point in deterministic nonsmooth and nonconvex optimization. For any deterministic algorithms that are accessible to both $1^\textnormal{st}$ and $0^\textnormal{th}$ oracles, we prove the dimension-dependent lower bound of $\Omega(d)$ on the complexity of finding a $(\delta, \epsilon)$-Goldstein stationary point for any fixed and finite dimension $d$ when $\delta, \epsilon > 0$ are smaller than some constants. Compared to the dimension-independent upper bounds for randomized algorithms that are accessible to $1^\textnormal{st}$ and $0^\textnormal{th}$ oracles, our results highlight the necessity of randomization in nonsmooth nonconvex optimization. Furthermore, we demonstrate the importance of $0^\textnormal{th}$ oracle by proving that any deterministic algorithm with only $1^\textnormal{st}$ oracle can not find an approximate Goldstein stationary point within the finite number of iterations up to some small constant tolerances. Finally, we propose a deterministic smoothing approach that achieves the smoothness parameter that is exponential in a certain parameter $M > 0$, and develop a deterministic algorithm with dimension-independent complexity bound of $\tilde{O}(M\delta^{-1}\epsilon^{-3})$. Future directions include the investigation of lower bound with improved dependence on $(d, \delta^{-1}, \epsilon^{-1})$, and the development of practical algorithms for nonsmooth nonconvex optimization. In particular, we have the following open problem:
\begin{quote}
\textbf{Open problem:} \textit{Is there a deterministic algorithm for nonsmooth nonconvex optimization with running time $\mathrm{poly}(d, 1/\epsilon, 1/\delta$) under a Lipschitz condition?}
\end{quote}

% Acknowledgements should only appear in the accepted version.
\section*{Acknowledgments}
The authors thank Guy Kornowski for pointing out an error in the first version. This work was supported in part by the Mathematical Data Science program of the Office of Naval Research under grant number N00014-18-1-2764 and by the Vannevar Bush Faculty Fellowship program
under grant number N00014-21-1-2941.
%%%%%%%%%%%%%%%%%%%%%%%%%%%%%%%%%%%%%%%%%%%%%%%%%%%%%%%%%%%%%%%%%%%%%

\bibliographystyle{plainnat}
\bibliography{ref}

\end{document}